\RequirePackage{fix-cm}

\documentclass[smallextended]{svjour3}       
\smartqed

\usepackage{graphicx}

\usepackage[table]{xcolor}
\usepackage{epstopdf}
\usepackage{float}
\usepackage{threeparttablex,color,soul,colortbl,hhline,rotating,booktabs,longtable,multirow,makecell}
\definecolor{lightgray}{gray}{0.95}

\usepackage{amssymb,amsfonts,amstext,amsmath}
\usepackage{mathptmx}
\usepackage{mathtools}

\usepackage{xcolor}
\usepackage{caption}
\usepackage{cite}
\usepackage{todonotes}
\usepackage{bm}
\usepackage{enumerate}
\usepackage{mathrsfs}
\usepackage{graphicx,geometry}
\usepackage{listings}

\usepackage[linkcolor=blue, urlcolor=blue, citecolor=blue,
colorlinks]{hyperref}

\DeclareMathSymbol{\mrq}{\mathord}{operators}{`'}

\usepackage[utf8]{inputenc}

\usepackage{amssymb,amsfonts,amstext,amsmath}
\usepackage{mathptmx}
\usepackage{cite,url}
\usepackage{graphicx,geometry}
\usepackage[title]{appendix}

\usepackage{multirow}

\makeatletter
\newcommand{\thickhline}{%
    \noalign {\ifnum 0=`}\fi \hrule height 1pt
    \futurelet \reserved@a \@xhline
}
\numberwithin{equation}{section}

\journalname{Soft Computing}

\begin{document}
\title{Generalized Hukuhara Hadamard Derivative of Interval-valued Functions and Its Applications 
	to Interval Optimization
}

\titlerunning{$gH$-Hadamard Derivative of Interval-valued Functions and Its Applications 
	to Interval Optimization}

\author{Ram Surat Chauhan \and Debdas Ghosh \and Qamrul Hasan Ansari}

\authorrunning{Chauhan, R. S. et al.}

\institute{Ram Surat Chauhan \at  \email{rschauhan.rs.mat16@itbhu.ac.in}\\
	Corresponding Author\\
             Department of mathematical sciences, Indian Institute of technology (BHU), Varanasi--221005, India
            \and
           Debdas Ghosh \at
           \email{debdas.mat@iitbhu.ac.in} \\
            Department of mathematical sciences, Indian Institute of technology (BHU), Varanasi--221005, India
             \and
           Qamrul Hasan Ansari\at
           \email{qhansari@gmail.com} \\
          Department of Mathematics, Aligarh Muslim University, Aligarh 202 002, India \\
          Department of Mathematics and Statistics, King Fahd University of Petroleum \& Minerals, Dhahran, Saudi Arabia
}

\date{Received: date / Accepted: date}

\maketitle

\begin{abstract}
  In this article, we study the notion of $gH$-Hadamard derivative for interval-valued functions (IVFs) 
  and its applications to interval optimization problems (IOPs). 
  It is shown that the existence of $gH$-Hadamard derivative implies 
  the existence of $gH$-Fr\'{e}chet derivative and vise-versa. 
  Further, it is proved that the existence of $gH$-Hadamard derivative implies the existence of 
  $gH$-continuity of IVFs. 
  We found that the composition of a Hadamard differentiable real-valued function and 
  a $gH$-Hadamard differentiable IVF is $gH$-Hadamard differentiable. 
  Further, for finite comparable IVF, we prove that the $gH$-Hadamard derivative of 
  the maximum of all finite comparable IVFs is the maximum of their $gH$-Hadamard derivative. 
  The proposed derivative is observed to be useful to check the convexity of an IVF and 
  to characterize efficient points of an optimization problem with IVF. 
  For a convex IVF, we prove that if at a point the $gH$-Hadamard derivative does not dominate to zero, 
  then the point is an efficient point. Further, it is proved that at an efficient point, 
  the $gH$-Hadamard derivative does not dominate zero and also contains zero. 
  For constraint IOPs, we prove an extended Karush-Kuhn-Tucker condition by using the proposed derivative. 
  The entire study is supported by suitable examples.
		\end{abstract}

		


	
	\section{Introduction}\label{sec:introd}
	

	In the study of general behavior of a real-world problem, such as static or dynamic, deterministic or probabilistic, 
	linear or nonlinear, convex or nonconvex, etc., several mathematical tools have been developed. 
	In many cases, the knowledge about the underlying parameters, which influences the system's mathematical behavior, 
	is imprecise or uncertain. Generally, one cannot measure the parameters affected by imprecision or 
	uncertainties with exact values. In such situations, the parameters cannot be modeled by a real number. 
	We usually overcome this deficiency by using fuzzy sets, interval, or stochastic values. 
	Interval analysis is based on representing an uncertain variable as an interval, 
	which is a natural way of incorporating the uncertainties of parameters. 
	As mathematical functions play a crucial role in modeling realistic problems, 
	we analyze a special derivative of interval-valued functions (IVFs) in this article.
	
	Three important aspects of a function are monotonicity, convexity and differentiability. 
	In the study of monotonicity and convexity \cite{ansari2013generalized} of an IVF, 
	an appropriate ordering of intervals is the prime issue. 
	Unlike the real numbers, intervals are not linearly ordered. 
	Due to which the whole paradigm of analyzing an IVF changes and the development of 
	calculus for IVF is not just trivial extensions of the corresponding counterpart 
	for conventional real-valued functions.
	The same reason makes the development of optimization with IVFs difficult 
	since the very optimality notion requires an ordering of the function values.
	
	Most often optimization with IVFs \cite{Chalco2013-2,Ghosh2019derivative,Stefanini2009} 
	have been analyzed with respect to a partial ordering \cite{Ishibuchi1990}. 
	Some  researchers \cite{Bhurjee2016,Ghosh2017spc,Kumar2021} used ordering relations of intervals based 
	on the parametric comparison of intervals. 
	In \cite{Costa2015}, an ordering relation of intervals is defined by a bijective map 
	from the set of intervals to $\mathbb{R}^2$. 
	However, the ordering relations  of intervals \cite{Bhurjee2016,Ghosh2017spc,Costa2015} 
	can be derived from the relations described in \cite{Ishibuchi1990}. 
	Recently, Ghosh et al. \cite{ghosh2020ordering} investigated variable ordering relations 
	for intervals and used them in interval optimization problems (IOPs).
	
	To observe the properties of an IVF, calculus plays an essential role. 
	Initially, to develop the calculus of IVFs, Hukuhara  \cite{Hukuhara1967} introduced 
	the concept of differentiability of IVFs with the help of $H$-difference of intervals. 
	However, the definition of Hukuhara differentiability ($H$-differentiability) is found to be restrictive 
	(see \cite{Chalco2013-2}). To remove the deficiencies of $H$-differentiability, 
	Bede and Gal \cite{Bede2005} defined strongly generalized derivative ($G$-derivative) for IVFs and derived 
	a Newton-Leibnitz-type formula. In order to formulate the mean-value theorem for IVFs, 
	Markov \cite{Markov1979} introduced a new concept of difference of intervals and 
	defined differentiability of IVFs by using this difference. 
	In \cite{Stefanini2009}, Stefanini and Bede defined the generalized Hukuhara differentiability 
	($gH$-differentiability) of IVFs by using the concept of generalized Hukuhara difference. 
	In defining the calculus of IVFs, the concepts of $gH$-derivative, $gH$-partial derivative, $gH$-gradient, 
	and $gH$-differentiability for IVFs have been developed in \cite{Ghosh2016newton,Stefanini2009,Stefanini2019}.
	
	To derive a Karush-Kuhn-Tucker (KKT) condtion for IOPs, Guo et al.\ \cite{Guo2019} 
	defined $gH$-symmetric derivative for IVFs. Ghosh \cite{Ghosh2016} analyzed the notion 
	of $gH$-differentiability of multi-variable IVFs to propose the Newton method for IOPs. 
	The concept of second-order differentiability of IVFs is introduced by Van \cite{Van2015} 
	to study the existence of a unique solution of interval differential equations. 
	Lupulescu \cite{Lupulescu2013} defined delta generalized Hukuhara differentiability on time scales by 
	using $gH$-difference. Chalco et al.\ \cite{Chalco2011} introduced the concept of $\pi$-derivative for 
	IVFs that generalizes Hukuhara derivative and $G$-derivative, and proved that this derivative 
	is equivalent to $gH$-derivative. 
	In \cite{Stefanini2014}, Stefanini and Bede defined level-wise $gH$-differentiability and 
	generalized fuzzy differentiability by LU-parametric representation for fuzzy-valued functions. 
	Kalani et al.\ \cite{Kalani2016} analyzed the concept of interval-valued fuzzy derivative for 
	perfect and semi-perfect interval-valued fuzzy mappings to derive a method for solving interval-valued 
	fuzzy differential equations using the extension principle. 
	Recently, Ghosh et al.\ \cite{Ghosh2019derivative} and Chauhan et al.\ \cite{Chauhand2021} have provided the idea of 
	$gH$-directional derivative, $gH$-G\^{a}teaux derivative,  $gH$-Fr\'echet derivative, and $gH$-Clarke derivative of 
	IVFs to derive the optimality conditions for IOPs.

	Despite of many attempts to develop calculus for IVFs, the existing ideas are not adequate to retain two most 
	important features of classical differential calculus---linearity of the derivative with respect to the direction 
	and the chain rule. Although Ghosh et al.\ \cite{Ghosh2019derivative} proposed some optimality conditions for IOPs 
	by using $gH$-directional and $gH$-G\^{a}teaux derivatives, but these derivatives are not sufficient to preserve 
	the continuity of IVFs (see Example 5.1 of \cite{Ghosh2019derivative}) and chain rule for the composition of IVFs 
	(see Example \ref{ex31} of this article). Even though $gH$-Fr\'echet derivative in \cite{Ghosh2019derivative} 
	preserves linearity and continuity but it does not hold the chain rule for the composition of IVFs whose lower 
	and upper functions are equal at each points (see example for Proposition 3.5 \cite{Shapiro1990}). 
	With the help of the derivative of lower and upper functions, some articles \cite{wu-kkt-iop-2009,zhang2014kkt,Ren2017} 
	reported KKT condition to characterize efficient solutions of constraint IOPs. 
	However, the derivative used in \cite{wu-kkt-iop-2009,zhang2014kkt,Ren2017} are very restrictive because this 
	derivative is very difficult to calculate even for very simple IVF (see Example 1 of \cite{Cano2015}). 
	However, in this article, we derive KKT condition of constraint IOPs by $gH$-Hadamard derivative which 
	do not depend on the existence of the Hadamard derivative of lower and upper functions. 
	Also, proposed derivative retains the linearity of the derivative with respect to direction, 
	the existence of continuity as well as the chain rule of derivative.
	
	\subsection{Motivation and Contribution}
	
	In conventional nonsmooth optimization theory, one of the mostly used idea of derivative is Hadamard derivative 
	which is applied to characterize optimal solutions. 
	An explicit expression of the derivative of an extremum with respect to parameters can be obtained with the help of Hadamard derivative. So, it works well for most differentiable optimization problems including convex or concave problems. Correspondingly, in interval analysis and interval optimization, we expect to have a notion of the Hadamard derivative for interval-valued functions. In addition, from the literature on the analysis of IVFs, one can notice that the study of Hadamard derivative for IVFs have not been developed so far. However, the basic properties of this derivative might be beneficial for characterizing and capturing the optimal solutions of IOPs.
	
	In this article, we define $gH$-Hadamard derivative of IVFs. It is proved that if an IVF is $gH$-Hadamard differentiable, then IVF is $gH$-continuous. By using the proposed concept of $gH$-Hadamard derivative, we prove that a $gH$-Fr\'{e}chet differentiable IVF is $gH$-Hadamard differentiable and vise-versa. Further, we characterize the convexity of IVFs with the help of $gH$-Hadamard derivative. Besides, with the help of $gH$-Hadamard derivative, we provide a necessary and sufficient condition for characterizing the efficient solutions to IOPs. Further, for constraint IOPs, we derive the extended KKT necessary and sufficient condition to characterize the efficient solutions.
	
	\subsection{Delineation}
	
	The rest of the article is demonstrated in the following sequence. 
	The next section covers some basic terminologies and notions of convex analysis and interval analysis, 
	followed by the convexity and calculus of IVFs. 
	Also, a few properties of intervals, $gH$-directional and $gH$-Fr\'{e}chet derivatives of an IVF are discussed in Section \ref{section1}. In Section \ref{section3}, we define the $gH$-Hadamard derivative of IVF and observe that the existence of $gH$-Hadamard derivative implies the existence of $gH$-Fr\'{e}chet derivative and vise-versa. Further, it is found that the existence of $gH$-Hadamard derivative implies $gH$-continuity. In the same section, it is shown that proposed derivatives are useful to check the convexity of an IVF. In Section \ref{section5}, we prove that at a point, in which $gH$-Hadamard derivative does not dominate zero is an efficient point of an IVF. Further, it is observed that at an efficient point of IVF, $gH$-Hadamard derivative must contain zero. In Section \ref{section5d}, a few properties of the cone of descent direction and cone of feasible direction are given. Also, we prove the extended necessary and sufficient optimality condition for constraint IOPs in the same section.  Finally, the last section concludes and draws future scopes of the study.

	
	\section{\textbf{Preliminaries and Terminology}}\label{section1}
	

	This section is devoted to some basic notions on intervals. 
	Also, we present basic convexity and calculus of IVFs which will be used throughout the paper. 
	We also use the following notations.
	\begin{itemize}
		\item $\mathcal{X}$ is a  real normed linear space with the norm $\|\cdot\|$
		
		\item $\mathcal{S}$ is a nonempty subset of $\mathcal{X}$
		
		\item $\mathbb{R}$ denotes the set of real numbers
		
		\item $\mathbb{R}_{+}$ denotes the set of nonnegative real numbers
		
		\item $I(\mathbb{R})$ is the set of all compact intervals (that is, closed and bounded intervals)
		\item $I(\mathbb{R})^n$ is the set of vectors whose componants are compact intervals
	\end{itemize}

	\subsection{Arithmetic of Intervals and their Dominance Relation}\label{ssai}
	Throughout the article, we denote the elements of $I(\mathbb{R})$ by bold capital letters: ${\textbf A}, {\textbf B}, {\textbf C}, \ldots $. We represent an element $\textbf{A}$ of $I(\mathbb{R})$ in its interval form with the help of the corresponding small letter in the following way:
	\[ \textbf{A} = [\underline{a}, \overline{a}],~\text{where}~\underline{a}~\text{and}~\overline{a}~\text{are real numbers such that}~\underline{a} \leq \overline{a}.\]
	
	Let $\textbf{A}, \textbf{B} \in I(\mathbb{R})$ and $\lambda \in \mathbb{R}$. Moore's \cite{Moore1966,Moore1987} interval addition, subtraction, product, division and scalar multiplication are denoted by $\textbf{A} \oplus \textbf{B},~ \textbf{A} \ominus \textbf{B},~ \textbf{A} \odot \textbf{B},~ \textbf{A} \oslash \textbf{B}$, and $\lambda \odot \textbf{A}$, respectively. In defining $\textbf{A} \oslash \textbf{B}$, it is assumed that $0\notin \textbf{B}$.
	
	Since $\textbf{A}\ominus\textbf{A}\neq\textbf{0}$ for any nondegenerate interval (whose lower and upper limits are not equal) $\textbf{A}$, we use the following concept of difference of intervals in this article.
	
	\begin{definition}\cite{Stefanini2008} Let $\textbf{A}$ and $\textbf{B}$ be two elements of $I(\mathbb{R})$. The    \emph{$gH$-difference} between $\textbf{A}$ and $\textbf{B}$, denoted by     $\textbf{A} \ominus_{gH} \textbf{B}$, is defined by the interval             $\textbf{C}$ such that
		\[ \textbf{A} =  \textbf{B} \oplus  \textbf{C} \quad \text{ or } \quad \textbf{B} =     \textbf{A} \ominus \textbf{C}.\]
		It is to be noted that for $\textbf{A} = \left[\underline{a},~\overline{a}\right]$ and $\textbf{B} = \left[\underline{b},~\overline{b}\right]$,
		\[ \textbf{A} \ominus_{gH} \textbf{B} = \left[\min\{\underline{a}-\underline{b}, \overline{a} - \overline{b}\},~ \max\{\underline{a}-\underline{b}, \overline{a} - \overline{b}\}\right] \text{ and } \textbf{A} \ominus_{gH} \textbf{A} = \textbf{0}.\]
	\end{definition}
	
	In the following, we provide a domination relation for intervals based on a \emph{minimization} type optimization problems: a \emph{smaller value} is \emph{better}.
	
	\begin{definition}\cite{Ishibuchi1990} \label{interval_dominance} Let $\textbf{A} = [\underline{a}, \overline{a}]$ and $\textbf{B} = [\underline{b}, \overline{b}]$ be two intervals in $I(\mathbb{R})$.
		\begin{enumerate}[(i)]
			\item $\textbf{B}$ is said to be \emph{dominated by} $\textbf{A}$ if $\underline{a}~\leq~ \underline{b}$ and $\overline{a}~\leq~\overline{b}$, and then we write $\textbf{A}~\preceq~ \textbf{B}$.
			
			\item $\textbf{B}$ is said to be \emph{strictly dominated by} $\textbf{A}$ if either `$\underline{a} \leq \underline{b}$  and $\overline{a} < \overline{b}$' or `$\underline{a} < \underline{b}$  and $\overline{a} \leq \overline{b}$', and then we write $\textbf{A}\prec \textbf{B}$.
			
			\item If $\textbf{B}$ is not dominated by $\textbf{A}$, then we write $\textbf{A}\npreceq \textbf{B}$; if $\textbf{B}$ is not strictly dominated by $\textbf{A}$, then we write $\textbf{A}\nprec \textbf{B}$.
			
			\item If $\textbf{A}$ is dominated by $\textbf{B}$ or $\textbf{B}$ is dominated by $\textbf{A}$, then $\textbf{A}$ and $\textbf{B}$ are said to be \emph{comparable}.
			
			\item If $\textbf{A}\npreceq \textbf{B}$ and $\textbf{B}~\npreceq~ \textbf{A}$, then we say that \emph{none of $\textbf{A}$ and $\textbf{B}$ dominates the other}, or $\textbf{A}$ and $\textbf{B}$ are \emph{not comparable}.
		\end{enumerate}
	\end{definition}

	Notice that if $\textbf{B}$ is strictly dominated by $\textbf{A}$, then $\textbf{B}$ is dominated by $\textbf{A}$. Moreover, if $\textbf{B}$ is not dominated by $\textbf{A}$, then $\textbf{B}$ is not strictly dominated by $\textbf{A}$.

	\begin{lemma}\label{forfrechet} For $\textbf{A}~\text{and}~\textbf{B}~\text{in}~ I(\mathbb{R})$,
		\begin{enumerate}[\rm(i)]
			\item \label{31} If $\textbf{B} \nprec \textbf{0}$ and $\textbf{B} \preceq \textbf{A}$,  then $\textbf{A} \nprec\textbf{0} $,
			
			\item \label{32} If	$\textbf{A}\oplus \textbf{B} \nprec \textbf{0} $ and $\textbf{B}\preceq \textbf{0}$, then $\textbf{A}\nprec \textbf{0}$.
		\end{enumerate}
	\end{lemma}
	
	\begin{proof}
		See \ref{nbe}
	\end{proof}

	\begin{definition}\cite{Chauhan2021}\label{better} Let $\textbf{A} = [\underline{a}, \bar{a}] = \{a(t) ~:~ a(t) = \underline{a} + t (\bar a - \underline{a}), 0\le t \le 1\}$ and  $\textbf{B} = [\underline{b}, \bar{b}] = \{b(t) ~:~ b(t) = \underline{b} + t (\bar b - \underline{b}), 0\le t \le 1\}$ be two elements of $I(\mathbb{R})$. Then, $\textbf{B}$ is said to be \emph{better strictly dominated by} by $\textbf{A}$ if $a(t) < b(t)$ for all $t \in [0, 1]$, and then we write $\textbf{A} < \textbf{B}$.
	\end{definition}

	\begin{lemma}\textnormal{\cite{Chauhan2021}} \label{rr} Let                      $\textbf{A}=[\underline{a}, \overline{a}]$ and $\textbf{B}=[\underline{b},   \overline{b}]$ be in $I(\mathbb{R})$. Then $\textbf{A}<\textbf{B}$   if and only if $\underline{a}<\underline{b}$ and                            $\overline{a}<\overline{b}$.
	\end{lemma}	
	
	\begin{definition}\cite{Moore1966}\label{irnorm} A       function ${\lVert \cdot \rVert}_{I(\mathbb{R})} : I(\mathbb{R}) \rightarrow   \mathbb{R}_{+}$ defined by
		\[ {\lVert \textbf{A} \rVert}_{I(\mathbb{R})}  = \max \{|\underline{a}|,     |\bar{a}|\},~~~~\text{for all}~\textbf{A}= [\underline{a}, \overline{a}] \in I(\mathbb{R}),\]
		is called a \emph{norm on} $I(\mathbb{R})$.
	\end{definition}
	
	\begin{definition}\label{max} For two    comparable intervals $\textbf{A}$ and $\textbf{B}$ of $I(\mathbb{R})$ with   $\textbf{A}\preceq \textbf{B}$, their \emph{maximum} is $\max\{\textbf{A},   \textbf{B}\} = \textbf{B}.$
	\end{definition}

	\subsection{Convexity and Calculus of IVFs}
	A function $\textbf{F} : \mathcal{S} \to I(\mathbb{R})$ is known as an IVF. For each $x \in \mathcal{S}$, $\textbf{F}$ can be presented by the following interval:
	\[\textbf{F}(x)=\left[\underline{f}(x),\ \overline{f}(x)\right],\]
	where $\underline{f}$ and $\overline{f}$ are real-valued functions on $\mathcal{S}$ such that $\underline{f}(x) \leq \overline{f}(x)$ for all $x \in \mathcal{S}$. Also, $\textbf{F}$ is said to be degenerate IVF if $\underline{f}(x) =\overline{f}(x)$ for all $x\in \mathcal{S}$. \\

	If $\mathcal{S}$ is convex, then the IVF $\textbf{F}$ is said to be convex \cite{Wu2007} on $\mathcal{S}$ if for any $x_1$, $x_2\in\mathcal{S}$,
	\[ \textbf{F}(\lambda_1 x_1+\lambda_2 x_2)\preceq
	\lambda_1\odot\textbf{F}(x_1)\oplus\lambda_2\odot\textbf{F}(x_2) \text{ for all } \lambda_1,\lambda_2\in[0, 1] \text{ with } \lambda_1+\lambda_2=1.\]
	The IVF $\textbf{F} : \mathcal{S} \rightarrow I(R)$ is said to be \emph{$gH$-continuous} \cite{Ghosh2016newton} at a point $\bar{x}$ of $\mathcal{S}$ if
	\[
	\lim_{\substack{\lVert d \rVert\rightarrow 0 \\  \bar{x}+d \in \mathcal{S}}}\left(\textbf{F}(\bar{x}+d)\ominus_{gH}\textbf{F}(\bar{x})\right)=\textbf{0}.
	\]
	If $\textbf{F}$ is $gH$-continuous at each point $x$ in $\mathcal{S}$, then $\textbf{F}$ is said to be $gH$-continuous on $\mathcal{S}$.
	
	\begin{lemma}\label{lc1} \textnormal{\cite{Wu2007}} $\textbf{F}$ is a convex IVF on   a convex set $\mathcal{S} \subseteq \mathcal{X}$ if and only if             $\underline{f}$ and $\overline{f}$ are convex on $\mathcal{S}$.
	\end{lemma}
	
	\begin{definition} \cite{Ghosh2019derivative} \label{linear}
		Let $\mathcal{S}$ be a linear subspace of $\mathcal{X}$. The function $\textbf{F}: \mathcal{S} \rightarrow I(\mathbb{R})$ is said to be \emph{linear} if
		\begin{enumerate}[(i)]
			\item $\textbf{F}(\lambda x)=\lambda\odot\textbf{F}(x)~ \text{for all}~x\in \mathcal{S}~\text{and all}~\lambda \in \mathbb{R}$, and
			
			\item  for all $x,~y\in \mathcal{S}$,
			\[\textquoteleft \text{either }\textbf{F}(x)\oplus\textbf{F}(y) = \textbf{F}(x+y) \mrq \quad \text{ or } \quad ~\textquoteleft \text{none of }\textbf{F}(x)\oplus\textbf{F}(y) \text{ and } \textbf{F}(x+y) \text{ dominates the other \textquoteright. } \]
		\end{enumerate}
	\end{definition}

	\begin{lemma}\label{nb} Let $\mathcal{S}$ be a linear subspace of $\mathcal{X}$ and $\textbf{F}: \mathcal{S} \to I(\mathbb{R})$ be a linear IVF. Then, the following results hold.
		\begin{enumerate}[\rm(i)]
			\item \label{nb1} If $\textbf{F}(x) \nprec \textbf{0}$ for all $x\in \mathcal{S}$, then $\textbf{0}$ and $\textbf{F}(x)$ are not comparable.
			
			\item \label{nb2} If $\textbf{F}(x) \preceq \textbf{0}$ for all $x\in \mathcal{S}$, then $\textbf{F}(x)= \textbf{0}$.
		\end{enumerate}
	\end{lemma}
	
	\begin{proof}
		See \ref{anb}
	\end{proof}

	\begin{definition}\label{efficient_point_def}\cite{Ghosh2019derivative} Let $\mathcal{S}$ be a nonempty subset of $\mathbb{R}^n$ and $\textbf{F}: \mathbb{R}^n \rightarrow I(\mathbb{R})$ be an IVF. A point $\bar{x}\in \mathcal{S}$ is said to be an \emph{efficient point} of the IOP
		\begin{equation}\label{IOP} \displaystyle \min_{x \in \mathcal{S}\subset \mathbb{R}^n} \textbf{F}(x)
		\end{equation}
		if $\textbf{F}(x)\nprec\textbf{F}(\bar{x})$ for all $x\in \mathcal{S}.$
	\end{definition}

	\begin{definition}\label{ddd} \cite{Delfour2012} Let $f$ be a real-valued function on a nonempty subset $\mathcal{S}$ of $\mathcal{X}$. Let $\bar{x} \in \mathcal{S}$ and $h \in \mathcal{X}$. If the limit
		\[
		\lim_{\lambda \to 0+}\frac{1}{\lambda} \left(f(\bar{x}+\lambda h)- f(\bar{x})\right)
		\]
		exists finitely, then the limit is said to be \emph{directional derivative} of $f$ at $\bar{x}$ in the direction $h$, and it is denoted by $f_\mathscr{D}(\bar{x})(h)$.
	\end{definition}

	\begin{definition}\label{frechet_derivative} \cite{Ghosh2019derivative} Let $\mathcal{S}$ be a nonempty open subset of $\mathcal{X}$, $\textbf{F}: \mathcal{S} \rightarrow I(\mathbb{R})$ be an IVF and $\bar{x} \in \mathcal{S}$. Suppose that there exists a $gH$-continuous and linear mapping $\textbf{G}:\mathcal{X}\rightarrow I(\mathbb{R})$ with the following property
		\[\lim_{{\lVert h \rVert} \to 0} \frac{{\lVert \textbf{F}(\bar{x}+h) \ominus_{gH} \textbf{F}(\bar{x}) \ominus_{gH} \textbf{G}(h) \rVert}_{I(\mathbb{R})}} {{\lVert h \rVert} } = 0,
		\]
		then $\textbf{G}$ is said to be  \emph{$gH$-Fr\'{e}chet derivative} of $\textbf{F}$ at $\bar{x}$, and we write $\textbf{G}=\textbf{F}_\mathscr{F}(\bar{x})$.
	\end{definition}


	\section{\textbf{$gH$-Hadamard Derivative of Interval-valued Functions}} \label{section3}
	
	In this section, we present the concept of  $gH$-Hadamard derivative for IVFs. It is worth noting that if an IVF $\mathbf{F}$ has a $gH$-Hadamard derivative at $\bar{x}$, then $\mathbf{F}$ must be  $gH$-continuous at $\bar{x}$.
	
	\begin{definition}\label{derivative}
		Let $\textbf{F}$ be an IVF on a nonempty subset $\mathcal{S}$ of $\mathcal{X}$. For   $\bar{x}\in \mathcal{S}$ and $v \in \mathcal{X}$, if the limit
		\[ \textbf{F}_\mathscr{H}(\bar{x})(v)\coloneqq
		\lim_{\substack{%
				\lambda \to 0+\\
				h \to v}}\frac{1}{\lambda}\odot\left(\textbf{F}(\bar{x}+\lambda h)\ominus_{gH}\textbf{F}(\bar{x})\right)
		\]
		exists and $\textbf{F}_\mathscr{H}(\bar{x})$ is a linear IVF from $\mathcal{X}$ to $I(\mathbb{R})$, then $\textbf{F}_\mathscr{H}(\bar{x})(v)$ is called\emph{ $gH$-Hadamard derivative} of $\textbf{F}$ at $\bar{x}$ in the direction $v$. If this limit exists for all $v \in \mathcal{X}$, then $\textbf{F}$ is said to be \emph{$gH$-Hadamard differentiable} at $\bar{x}$.
	\end{definition}
	
	\begin{remark}\label{re1} The limit $\textbf{F}_\mathscr{H}(\bar{x})(v)$ exists if for all sequences $\{\lambda_n\}$ and $\{h_n\}$ with $\lambda_n>0$ for all $n$ such that $\lim_{n\to \infty} \lambda_n = 0$, $\lim_{n\to \infty} h_n = v$,
		\[\lim_{n \to \infty} \frac{1}{\lambda_n}\odot\left(\textbf{F}(\bar{x}+\lambda_n h_n)\ominus_{gH}\textbf{F}(\bar{x})\right) \text{ exists and the limit value is a linear IVF on}~ \mathcal{S}.\]
	\end{remark}
	
	\begin{example}\label{ee1}
		Let $\mathcal{S}=\mathcal{X}=\mathbb{R}^n$ and consider the IVF $\textbf{F}(x) = {\lVert x \rVert}^2 \odot \textbf{C}$~\text{for all}~ $x \in \mathbb{R}^n$, where $\textbf{C} \in I(\mathbb{R})$. Then we calculate the $gH$-Hadamard derivative at $\bar{x}=0$ for $\textbf{F}$.\\
		For any $\bar{x}\in \mathcal{S}$ and $v \in \mathcal{X}$, we see that
		\begin{eqnarray*}\label{0_1}
			\lim_{\substack{%
					\lambda \to 0+\\
					h \to v}}\frac{1}{\lambda}\odot\left(\textbf{F}(\bar{x}+\lambda h)\ominus_{gH}\textbf{F}(\bar{x})\right)
			&=& \lim_{\substack{%
					\lambda \to 0+\\
					h \to v}}\frac{1}{\lambda}\odot \left({\lVert \bar{x}+\lambda h \rVert}^2 \odot \textbf{C}\ominus_{gH}{\lVert \bar{x} \rVert}^2 \odot \textbf{C}\right)\\
			&=& \lim_{\substack{%
					\lambda \to 0+\\
					h \to v}}\frac{1}{\lambda}\odot \left(\left( 2\bar{x}^{\top}(\lambda h) +{\lVert \lambda h \rVert}^2\right) \odot \textbf{C}\right)\\
			&=& 2\bar{x}^\top v \odot \textbf{C},~\text{by $gH$-continuity of}~\bar{x}^{\top}h \odot \textbf{C}.
		\end{eqnarray*}
		Hence, $\textbf{F}_\mathscr{H}(\bar{x})(v) = 2\bar{x}^{\top}v \odot \textbf{C}$ and $\textbf{F}_\mathscr{H}(\bar{x})$ is a linear IVF from $\mathcal{X}$ to $I(\mathbb{R})$. Therefore, $\textbf{F}$ is $gH$-Hadamard differentiable at $\bar{x}$ with $\textbf{F}_\mathscr{H}(\bar{x})(v)=2\bar{x}^{\top}v \odot \textbf{C}$ .
	\end{example}


	\begin{theorem}\label{th31}
		Let $\mathcal{X}=\mathbb{R}^n$, $\mathcal{S}$ be a nonempty subset of $\mathcal{X}$, $\textbf{F}$ be an IVF on $\mathcal{S}$ and $\bar{x}\in \mathcal{S}$. Then the following statements are equivalent:	\begin{enumerate}[\rm(i)]
			\item\label{th1} $\textbf{F}$ is $gH$-Fr\'{e}chet differentiable at $\bar{x}$.
			
			\item\label{th2}$\textbf{F}$ is $gH$-Hadamard differentiable at $\bar{x}$.
		\end{enumerate}
	\end{theorem}
	
	\begin{proof}
		(\ref{th1}) $\implies$ (\ref{th2}). Since $\textbf{F}$ is $gH$-Fr\'{e}chet differentiable at $\bar{x} \in \mathcal{S}$, there exists a $gH$-continuous and linear IVF $\textbf{G}$ such that
		\begin{align}
			& \notag \lim_{\lambda \to 0+} \frac{{\lVert \textbf{F}(\bar{x}+\lambda h) \ominus_{gH} \textbf{F}(\bar{x}) \ominus_{gH} \textbf{G}(\lambda h) \rVert}_{I(\mathbb{R})}} {{\lVert \lambda h \rVert}}=0, \quad \text{ for all } h\in \mathcal{X}\backslash\{\hat{0} \} \\
			\text{or, } & \lim_{\lambda \to 0+} \frac{1}{\lambda}\lVert \textbf{F}(\bar{x}+\lambda h) \ominus_{gH} \textbf{F}(\bar{x}) \ominus_{gH} \textbf{G}(\lambda h) \rVert_{I(\mathbb{R})}=0, \quad \text{ for all }  h\in \mathcal{X}\backslash\{\hat{0}\}. \label{equation_sub1}
		\end{align}
		Since $\textbf{G}$ is linear, and thus $\textbf{G}(\lambda ~ h) = \lambda \odot \textbf{G}(h)$, the equation (\ref{equation_sub1}) gives
		\begin{align*}
			& \lim_{\lambda \to 0+} \frac{1}{\lambda}\odot\left(\textbf{F}(\bar{x}+\lambda h) \ominus_{gH} \textbf{F}(\bar{x}) \ominus_{gH} \lambda \odot\textbf{G}(h)\right) = \textbf{0}, \quad \text{ for all } h\in \mathcal{X}\backslash\{\hat{0}\} \\
			\text{or, } & \lim_{\lambda \to 0+} \frac{1}{\lambda}\odot\left(\textbf{F}(\bar{x}+\lambda h) \ominus_{gH} \textbf{F}(\bar{x})\right)=\textbf{G}(h), \quad \text{ for all } h\in \mathcal{X}\backslash\{\hat{0}\}.
		\end{align*}
		Since $\textbf{G}$ is $gH$-continuous, we have
		\[	\lim_{\substack{%
				\lambda \to 0+\\
				h \to v}}\frac{1}{\lambda}\odot\left(\textbf{F}(\bar{x}+\lambda h)\ominus_{gH}\textbf{F}(\bar{x})\right) = \textbf{G}(v).
		\]
		Hence, $\textbf{F}$ is $gH$-Hadamard differentiable at $\bar{x}$.
		\newline \\
		\noindent (\ref{th2}) $\implies$ (\ref{th1}). As $\textbf{F}$ is $gH$-Hadamard differentiable at $\bar{x} \in \mathcal{S}$, $\textbf{F}_\mathscr{H}(\bar{x})(v)$ exists for all $v$ and $\textbf{F}_\mathscr{H}(\bar{x})$ is a linear IVF. Let
		\[Q(h)=  \frac{1}{\lVert h \rVert} \odot \left( \textbf{F}(\bar{x}+h) \ominus_{gH} \textbf{F}(\bar{x}) \ominus_{gH} \textbf{F}_\mathscr{H}(\bar{x})(h) \right),~h\neq \hat{0}.\]
		Consider a sequence $\{h_n\}$ converging to $0$. As $\mathcal{W}=\{h/\lVert h\rVert:  h \in \mathcal{X}, h \neq \hat{0}\}$ is a compact set, there exists a subsequences $\{h_{n_k}\}$ and a point $\bar{v}\in \mathcal{W}$ such that $w_{n_k}=\frac{h_{n_k}}{\lVert h_{n_k}\rVert}\to \bar{v} \in \mathcal{W}.$\\
		Note that the sequence $\{t_{n_k}\}$, defined by $t_{n_k}=\lVert h_{n_k}\rVert$, converges to $0$. Since $\textbf{F}_\mathscr{H}(\bar{x})(\bar{v})$ exists and\\ $\textbf{F}_\mathscr{H}(\bar{x})(w_{n_k}) \to \textbf{F}_\mathscr{H}(\bar{x})(\bar{v})$ as $k \to \infty$,  we have
		\[Q(h_{n_k})=\frac{1}{t_{n_k}}\odot\big(\textbf{F}(\bar{x}+t_{n_k} w_{n_k}) \ominus_{gH} \textbf{F}(\bar{x})\big)\ominus_{gH}\textbf{F}_\mathscr{H}(\bar{x})(w_{n_k}) \to \textbf{0} \text{ as } k \to \infty.\]
		This implies that $\lim_{k \to \infty}\lVert Q(h_{n_k})\rVert_{I(\mathbb{R})} = 0.$\\
		As $\{h_n\}$ is an arbitrarily chosen sequence that converges to $0$, $\lim_{\lVert h \rVert \to 0}\lVert Q(h)\rVert_{I(\mathbb{R})} = 0$. Hence, $\textbf{F}$ is $gH$-Fr\'{e}chet differentiable at $\bar{x}$.
	\end{proof}
	
	\begin{remark}
		If $\mathcal{X}$ is infinite dimensional, then Theorem \ref{th31} is not true. For instance, see Example 1 of \cite{Yu2013}. According to this example, there exists a degenerate IVF $\textbf{F}$ which is $gH$-Hadamard differentiable at $\bar{x}$ but not $gH$-Fr\'{e}chet differentiable at $\bar{x}$.
	\end{remark}

	\begin{theorem}\label{the1}
		Let $\mathcal{S}$ be a nonempty subset of $\mathcal{X}=\mathbb{R}^n$. If the function $\textbf{F}: \mathcal{S} \rightarrow I(\mathbb{R})$ has a $gH$-Hadamard derivative at $\bar{x} \in \mathcal{S}$, then the function $\textbf{F}$ is $gH$-continuous at $\bar{x}$.
	\end{theorem}
	
	\begin{proof}
		Since $\textbf{F}$ is $gH$-Hadamard differentiable at $\bar{x}\in \mathcal{S}$, $\textbf{F}$ is $gH$-Fr\'{e}chet differentiable at $\bar{x}$ by Theorem  \ref{th31}. Also, from Theorem 5.1 in \cite{Ghosh2019derivative}, the function $\textbf{F}$ is $gH$-continuous at $\bar{x}$.
	\end{proof}
	
	\begin{remark}
		The converse of Theorem \ref{the1} is not true.
		For instance, consider the $gH$-continuous IVF $\textbf{F}(x) = {\lVert x \rVert} \odot \textbf{C}~\text{for all}~x\in \mathbb{R}^n.$ Therefore, for any $v \in \mathbb{R}^n$ and $\bar{x} = 0$, we see that
		\[\lim_{\substack{%
				\lambda \to 0+\\
				h \to v}}\frac{1}{\lambda}\odot\left(\textbf{F}(\bar{x}+\lambda h)\ominus_{gH}\textbf{F}(\bar{x})\right)
		=\lVert v \rVert\odot \textbf{C}.\]
		Hence, the limit value is not a linear IVF on $\mathcal{S}$. Therefore, $\textbf{F}_\mathscr{H}(\bar{x})(h)$ does not exist.
	\end{remark}

	\begin{remark}\label{r2}
		By the definitions of $gH$-directional (Definition 3.1 in \cite{Ghosh2019derivative}), $gH$-G\^{a}teaux (Definition 4.3 in \cite{Ghosh2019derivative}) and $gH$-Hadamard (Definition \ref{derivative}) derivatives of IVF $\textbf{F}$, it is clear that if $\textbf{F}_\mathscr{H}(\bar{x})(h)$ exists, then $\textbf{F}_\mathscr{D}(\bar{x})(h)$ and $\textbf{F}_\mathscr{G}(\bar{x})(h)$ exist and they are equal to $\textbf{F}_\mathscr{H}(\bar{x})(h)$. However, the converse is not true. For instance, consider the IVF $\textbf{F}: \mathbb{R}^2 \rightarrow I(\mathbb{R})$ defined by
		\[\textbf{F}(x, y)=
		\begin{cases}
		\left( \frac{x^6}{(y-x^2)^2+x^8}\right) \odot [3,~9], & \text{if } (x,y) \neq (0, 0),\\
		\textbf{0}, & \text{otherwise}.
		\end{cases}
		\]
		For $\bar{x}=(0, 0)$ and arbitrary $h=(h_1, h_2)\in \mathbb{R}^2$, we have
		\[\lim_{\substack{%
				\lambda \to 0+\\}}\frac{1}{\lambda}\odot\left(\textbf{F}(\bar{x}+\lambda h)\ominus_{gH}\textbf{F}(\bar{x})\right)
		= \lim_{\substack{%
				\lambda \to 0+\\
			}}\frac{1}{\lambda}\odot\left(	\left( \frac{\lambda^6h_1^6}{(\lambda h_2-\lambda^2 h_1^2)^2+\lambda ^8 h_1^8}\right) \odot [3,~9] \right)\\
			= \textbf{0}.
			\]
			Hence, $\textbf{F}$ is $gH$-directional and $gH$-G\^{a}teaux differentiable at $\bar{x}$ with $\textbf{F}_\mathscr{D}(\bar{x})(h)= \textbf{F}_\mathscr{G}(\bar{x})(h)=\textbf{0}$.\\
			Let $\lambda_n = \frac{1}{n}$ and $h_n=(\frac{1}{n}, \frac{1}{n^3})~\text{for}~n \in \mathbb{N}$. Then, for $\bar{x}=(0, 0)$, we have
			\begin{equation}\label{eq1}
				\lim_{n\to \infty} \frac{1}{\lambda_n}\odot\left(\textbf{F}(\bar{x}+\lambda_n h_n)\ominus_{gH}\textbf{F}(\bar{x})\right) = \lim_{n\to \infty} n^5\odot [3, 9].
			\end{equation}
			Hence, $\textbf{F}_\mathscr{H}(\bar{x})(0)$ does not exist.
		\end{remark}
		
		\begin{theorem}\label{eth33}
			Let $\mathcal{S}$ be a nonempty convex subset of $\mathbb{R}^n$ and the IVF $\textbf{F}: \mathcal{S} \rightarrow I(\mathbb{R})$ has $gH$-Hadamard derivative at every $\bar{x} \in \mathcal{S}$. If the function $\textbf{F}$ is convex on $\mathcal{S}$, then
			\[
			\textbf{F}(v)\ominus_{gH}\textbf{F}(\bar{x}) \nprec\textbf{F}_\mathscr{H}(\bar{x})(v-\bar{x}), \quad \text{ for all } v\in \mathcal{S}.
			\]
		\end{theorem}
		
		\begin{proof}
			Since $\textbf{F}$ is convex on $\mathcal{S}$, for any $\bar{x},~h \in \mathcal{S}$ and $\lambda,~\lambda' \in (0,1]$ with $\lambda+\lambda'=1$, we have
			\begin{align*}
				&\textbf{F}(\bar{x}+\lambda (h-\bar{x}))=\textbf{F}(\lambda h+\lambda'\bar{x})\preceq\lambda\odot\textbf{F}(h)\oplus\lambda'\odot\textbf{F}(\bar{x})\\
				\implies	&	\textbf{F}(\bar{x}+\lambda (h-\bar{x}))\ominus_{gH}\textbf{F}(\bar{x})\preceq(\lambda\odot\textbf{F}(h)\oplus\lambda'\odot\textbf{F}(\bar{x}))\ominus_{gH}\textbf{F}(\bar{x})\\
				\implies&\textbf{F}(\bar{x}+\lambda (h-\bar{x}))\ominus_{gH}\textbf{F}(\bar{x})\preceq \lambda\odot(\textbf{F}(h)\ominus_{gH}\textbf{F}(\bar{x}))\\
				\implies&\frac{1}{\lambda}\odot(\textbf{F}(\bar{x}+\lambda (h-\bar{x}))\ominus_{gH}\textbf{F}(\bar{x}))\preceq\textbf{F}(h)\ominus_{gH}\textbf{F}(\bar{x}).
			\end{align*}
			From Theorem \ref{the1}, $\textbf{F}$ is $gH$-continuous. Thus, as $\lambda\to 0+$ and $h\to v$, we obtain
			\begin{equation}\label{gc}
				\textbf{F}_\mathscr{H}(\bar{x})(v-\bar{x})\preceq \textbf{F}(v)\ominus_{gH}\textbf{F}(\bar{x}), \quad \text{for all}~v\in \mathcal{S}.
			\end{equation}
			If possible, let
			\[
			\textbf{F}(v') \ominus_{gH} \textbf{F}(\bar{x}')\prec \textbf{F}_\mathscr{H}(\bar{x}')(v'-\bar{x}')~\text{for some}~v'\in \mathcal{X}.
			\]
			Then,
			\[
			\textbf{F}(v')\ominus_{gH}\textbf{F}(\bar{x}')\prec\textbf{F}_\mathscr{H}(\bar{x}')(v'-\bar{x}'),
			\]
			which contradicts (\ref{gc}). Hence,
			\[
			\textbf{F}(v)\ominus_{gH}\textbf{F}(\bar{x}) \nprec\textbf{F}_\mathscr{H}(\bar{x})(v-\bar{x}), \quad \text{ for all } v\in \mathcal{S}.
			\]
		\end{proof}

		\begin{remark}\label{nee1}
			The converse of Theorem \ref{eth33} is not true. For example, let us consider the IVF $\textbf{F}: \mathbb{R} \rightarrow I(\mathbb{R})$ defined by
			\[\textbf{F}(x)=[-4x^2, 6x^2].\]
			At $\bar{x}=0\in \mathbb{R}$, for arbitrary $v\in \mathbb{R}$, we have
			\begin{eqnarray*}
				&& \textbf{F}_\mathscr{H}(\bar{x})(v)=\lim_{\substack{%
						\lambda \to 0+\\
						h \to v}} \frac{1}{\lambda}\odot\left(\textbf{F}(\bar{x} + \lambda h) \ominus_{gH} \textbf{F}(\bar{x})\right)
				=\textbf{0}.
			\end{eqnarray*}
			Hence, $
			\textbf{F}(v) \ominus_{gH}\textbf{F}(\bar{x}) \nprec\textbf{F}_\mathscr{H}(\bar{x})(v-\bar{x})  \text{ for all } v\in \mathbb{R}.$ However, $\underline{f}$ is not convex on $\mathbb{R}$. Thus, from Lemma \ref{lc1}, $\textbf{F}$ is not convex on $\mathbb{R}$.
		\end{remark}
		
		\begin{remark}
			For a convex IVF \textbf{F} on $\mathcal{S}\subset \mathbb{R}^n$, the inequality \textquoteleft$\textbf{F}_\mathscr{H}(\bar{x})(v-\bar{x}) \ominus_{gH} \textbf{F}_\mathscr{H}(v)(v-\bar{x}) \preceq \textbf{0}$ for all $\bar{x}, v \in \mathcal{S}$\textquoteright~is not true. For instance,  consider the convex IVF $\textbf{F}: \mathbb{R} \rightarrow I(\mathbb{R})$ defined by
			\[\textbf{F}(x)= [x^2, 3x^2].\]
			At $\bar{x}\in \mathbb{R}$, for arbitrary $v\in \mathbb{R}$, we have $\textbf{F}_\mathscr{H}(\bar{x})(v-\bar{x}) = 2 \bar{x}(v-\bar{x})\odot [1, 3]$. For $\bar{x}=1$ and $v=2$, we obtain $\textbf{F}_\mathscr{H}(\bar{x})(v-\bar{x}) \ominus \textbf{F}_\mathscr{H}(v)(v-\bar{x})=[-10, 2]\npreceq \textbf{0}.$
		\end{remark}
		
		\begin{theorem}\label{th4444ed}
			Let $\textbf{F}: \mathbb{R}^n \rightarrow I(\mathbb{R})$ be an IVF and $\bar{x} \in \mathbb{R}^n$. Then, for a given direction $v\in \mathbb{R}^n$, the following statements are equivalent:
			\begin{enumerate}[\rm(i)]
				\item \label{th321ed} $\textbf{F}$ is $gH$-Hadamard differentiable at $\bar{x}$;
				
				\item \label{th322ed} There exists a linear IVF $\textbf{L}: \mathbb{R}^n \rightarrow I(\mathbb{R})$ such that for any path  $f: \mathbb{R} \to \mathbb{R}^n$ with $f(0)=\bar{x}$ for which $f_\mathscr{D}(0)(1)$ exists, we have
				\[(\textbf{F}\circ f)_\mathscr{D}(0)(1)= \textbf{L}(\bar{x})(v), \quad \text{ where } v=f_\mathscr{D}(0)(1).\]
				
			\end{enumerate}
		\end{theorem}
		
		\begin{proof}(\ref{th321ed}) $\implies$ (\ref{th322ed}). Let $\{\delta_n\}$ be a sequence of positive real numbers with $\delta_n \to 0^+$ and $h_n = \frac{1}{\delta_n}\left(f(\delta_n)-f(0)\right)$ for all $n \in \mathbb{N}$. Since $f_\mathscr{D}(0)(1)$ exists, we have
			\begin{equation}\label{e321ed}
				\lim_{n \to \infty} h_n = \lim_{n \to \infty} \frac{1}{\delta_n}\odot\left(f(\delta_n)- f(0) \right) = f_\mathscr{D}(0)(1)=v.
			\end{equation}
			If $\textbf{F}$ is $gH$-Hadamard differentiable at $\bar{x}$, then
			\begin{eqnarray*}
				\textbf{F}_\mathscr{H}(\bar{x})(v)&=&\lim_{n \to \infty}\frac{1}{\delta_n}\odot\left(\textbf{F}(\bar{x}+\delta_n h_n)\ominus_{gH}\textbf{F}(\bar{x})\right)\\
				&=& \lim_{n \to \infty}\frac{1}{\delta_n}\odot\left(\textbf{F}(f(\delta_n))\ominus_{gH}\textbf{F}(f(0))\right),\quad \text{ since } f(0)=\bar{x} \text{ and } h_n = \frac{1}{\delta_n}\left(f(\delta_n)-f(0)\right) \\
				&=& \lim_{n \to \infty}\frac{1}{\delta_n}\odot\left((\textbf{F} \circ f)(\delta_n)\ominus_{gH}(\textbf{F} \circ f)(0)\right).
			\end{eqnarray*}
			Hence, $(\textbf{F}\circ f)_\mathscr{D}(0)(1)=\textbf{F}_\mathscr{H}(\bar{x})(v)$. Due to the linearity of $\textbf{F}_\mathscr{H}(\bar{x})(v)$ on $\mathbb{R}^n$, by taking $\textbf{L}(\bar{x})(v)=\textbf{F}_\mathscr{H}(\bar{x})(v)$, we get the desired result.\\
			
			\noindent (\ref{th322ed}) $\implies$ (\ref{th321ed}). If possible, assume that $\textbf{F}$ is not $gH$-Hadamard differentiable at $\bar{x}$. Then, there exist sequences $h_n \to v$ and $\delta_n \to 0^+$ such that
			\begin{equation}\label{ethe36ed}
				\textquoteleft \text{ either }\lim_{n \to \infty}\frac{1}{\delta_n}\odot\left(\textbf{F}(\bar{x}+\delta_n h_n)\ominus_{gH}\textbf{F}(\bar{x})\right) \text{ does not exist\textquoteright \quad or \quad \textquoteleft limit value is not linear IVF on } \mathbb{R}^n \mrq .
			\end{equation}
			Since $h_n \to v$ and $\delta_n \to 0^+$, for every $\epsilon >0 $ there exist a natural number $N$ and a real number $a$ such that \begin{equation}\label{e322ed}
				\lVert h_n\rVert \leq a,\quad \lVert h_n -v\rVert < \epsilon, ~\text{and} \quad \delta_n < \epsilon/a~\text{for all} ~n> N. \end{equation}
			By using the sequences $\{h_n\}$ and $\{\delta_n\}$, we construct a function $f: \mathbb{R} \rightarrow \mathbb{R}^n$ as follows:
			\[ f(\delta) =
			\begin{cases}
			\bar{x} +\delta v,& \text{if $\delta \leq 0$},\\
			\bar{x}+\delta h_n, & \text{if $\delta_n \leq \delta < \delta_{n-1}, n\geq 2$},\\
			\bar{x}+\delta h_1, & \text{if $\delta \geq \delta_1.$}\\
			\end{cases}\]
			Thus the function $f$ yields $f(0)=\bar{x}$ and $f_\mathscr{D}(0)(1)=v$ (for details, see p. 92 in \cite{Delfour2012}). By hypothesis,\\ $(\textbf{F}\circ f)_\mathscr{D}(0)(1)$ exists and equals to $\textbf{L}(\bar{x})(v)$, where $v=f_\mathscr{D}(0)(1)$. From the construction of $f$, we have
			\begin{align*}
				& \lim_{n \to \infty}\frac{1}{\delta_n}\odot\left((\textbf{F} \circ f)(\delta_n)\ominus_{gH}(\textbf{F} \circ f)(0)\right) =\textbf{L}(\bar{x})(v)\\
				\text{or, }& \lim_{n \to \infty}\frac{1}{\delta_n}\odot\left(\textbf{F}(f(\delta_n))\ominus_{gH}\textbf{F}(f(0))\right) =\textbf{L}(\bar{x})(v)\\
				\text{or, }&\lim_{n \to \infty}\frac{1}{\delta_n}\odot\left(\textbf{F}(\bar{x}+\delta_n h_n)\ominus_{gH}\textbf{F}(\bar{x})\right) =\textbf{L}(\bar{x})(v),
			\end{align*}
			which contradicts to (\ref{ethe36ed}). Therefore, $\textbf{F}$ is $gH$-Hadamard differentiable at $\bar{x}$.
			
		\end{proof}

		\begin{theorem}\label{th34edd}$($\emph{Chain rule}$)$.
			Let $H:\mathbb{R}^m \rightarrow \mathbb{R}^n$ be a vector-valued function and $\textbf{F}:\mathbb{R}^n \rightarrow I(\mathbb{R})$ be an IVF. Assume that for a point $\bar{x}\in \mathbb{R}^m$ and direction $v\in \mathbb{R}^m$,
			\begin{enumerate}[\rm(a)]
				\item \label{ath341edd} $H_\mathscr{D}(\bar{x})(v)$ exists for all $v \in \mathbb{R}^m$, and
				
				\item \label{ath342edd} $\textbf{F} _\mathscr{H}(\bar{y})(z)$ exists, \quad where $\bar{y}=H(\bar{x})$ and $ z=H_\mathscr{D}(\bar{x})(v)$.
			\end{enumerate}
			Then,
			\begin{enumerate}[\rm(i)]
				\item\label{th341eedd} $(\textbf{F}\circ H)_\mathscr{D}(\bar{x})(v)$ exists and
				$(\textbf{F}\circ H)_\mathscr{D}(\bar{x})(v)= \textbf{F} _\mathscr{H}(\bar{y})(z)$
				
				\item \label{th342eedd} if $H _\mathscr{H}(\bar{x})(v)$ exists, then $(\textbf{F}\circ H)_\mathscr{H}(\bar{x})(v)$ exists and
				\[(\textbf{F}\circ H)_\mathscr{H}(\bar{x})(v)= \textbf{F} _\mathscr{H}(\bar{y})(\bar{z}), \quad \text{where}~\bar{y}=H(\bar{x}),~ \bar{z}=H_\mathscr{H}(\bar{x})(v).\]
			\end{enumerate}
		\end{theorem}
		
		\begin{proof}(\ref{th341eedd})
			For $\delta >0$, define
			\begin{equation}\label{eth341edd}
				\textbf{Q}(\delta)= \frac{1}{\delta}\odot\big(\textbf{F}(H(\bar{x}+\delta v))\ominus_{gH}\textbf{F}(H(\bar{x}))\big) \quad \text{and} \quad \theta(\delta)=\frac{1}{\delta}\big(H(\bar{x}+\delta v)-H(\bar{x})\big).
			\end{equation}
			Then,
			\begin{equation}\label{eth342edd}
				\textbf{Q}(\delta)=\frac{1}{\delta}\odot\big(\textbf{F}(H(\bar{x})+\delta \theta(\delta))\ominus_{gH}\textbf{F}(H(\bar{x}))\big).
			\end{equation}
			Since $\theta(\delta)\to H_\mathscr{D}(\bar{x})(v)$ as $\delta \to 0+$, from (\ref{eth341edd}), (\ref{eth342edd}) and the hypothesis (\ref{ath342edd}), we have
			\begin{eqnarray*}
				\textbf{F}_\mathscr{H}(\bar{y})(z)&=&	\lim_{\delta \to 0+}
				\frac{1}{\delta}\odot\left(\textbf{F}(H(\bar{x}+\delta v))\ominus_{gH}\textbf{F}(H(\bar{x}))\right) ,~\text{where}~\bar{y}=H(\bar{x}), z=H_\mathscr{D}(\bar{x})(v)\\
				&=& \lim_{
					\delta \to 0+} \frac{1}{\delta}\odot\left((\textbf{F} \circ H)(\bar{x}+\delta v)\ominus_{gH}(\textbf{F} \circ H)(\bar{x}))\right)\\
				&=& (\textbf{F}\circ H)_\mathscr{D}(\bar{x})(v).
			\end{eqnarray*}
			(\ref{th342eedd}) For $\delta >0$ and $h\in \mathbb{R}^m$, define
			\begin{equation}\label{eth343edd}
				\textbf{Q}'(\delta, h)= \frac{1}{\delta}\odot\left(\textbf{F}(H(\bar{x}+\delta h))\ominus_{gH}\textbf{F}(H(\bar{x}))\right) \text{ and } \Phi(\delta, h)=\frac{1}{\delta}\left(H(\bar{x}+\delta h)-H(\bar{x})\right).
			\end{equation}
			Then,
			\begin{equation}\label{eth344edd}
				\textbf{Q}'(\delta, h)=\frac{1}{\delta}\odot\left(\textbf{F}(H(\bar{x})+\delta \Phi(\delta,h))\ominus_{gH}\textbf{F}(H(\bar{x}))\right).
			\end{equation}
			Since $\Phi(\delta, h)\to H_\mathscr{H}(\bar{x})(v)~\text{as}~ \delta \to 0+~\text{and}~h\to v$, from (\ref{eth343edd}), (\ref{eth344edd}) and the hypothesis (\ref{ath342edd}), we have
			\begin{eqnarray*}
				\textbf{F}_\mathscr{H}(\bar{y})(\bar{k}) &=&	\lim_{\substack{%
						\delta \to 0+\\
						h \to v}} \frac{1}{\delta}\odot\big(\textbf{F}(H(\bar{x}+\delta h))\ominus_{gH}\textbf{F}(H(\bar{x}))\big),~\text{where}~\bar{y}=H(\bar{x}),~ \bar{z}=H_\mathscr{H}(\bar{x})(v)\\
				&=& \lim_{\substack{%
						\delta \to 0+\\
						h \to v}} \frac{1}{\delta}\odot\big(\textbf{F} \circ H)(\bar{x}+\delta h)\ominus_{gH}(\textbf{F} \circ H)(\bar{x}))\big )\\
				&=& (\textbf{F}\circ H)_\mathscr{H}(\bar{x})(v).
			\end{eqnarray*}
		\end{proof}
		
		The weaker assumption---the existence of $G_\mathscr{D}(\bar{x})(v)$ and $\textbf{F} _\mathscr{D}(\bar{y})(k)$ with $\bar{y}=G(\bar{x}), k=G_\mathscr{D}(\bar{x})(v)$---is not sufficient to prove Theorem \ref{th34edd}. For the proof of this theorem, we require a strong assumption (\ref{ath342edd}) of Theorem \ref{th34edd}. This is illustrated by the following example that the composition $\textbf{F}\circ G$, of a $gH$-G\^{a}teaux differentiable IVF $\textbf{F}$ and a G\^{a}teaux differentiable vector-valued function $G$, is not $gH$-G\^{a}teaux differentiable and even not $gH$-directional differentiable in any direction $v\neq 0$.
		
		\begin{example}\label{ex31}
			Consider the IVF $\textbf{F}:\mathbb{R}^2 \rightarrow I(\mathbb{R})$ defined by
			\[	\textbf{F}(x, y)=
			\begin{cases}
			\left( \frac{x^6}{(y-x^2)^2+x^8}\right) \odot [2,~6], & \text{if } (x,y) \neq (0, 0),\\
			\textbf{0}, & \text{otherwise},
			\end{cases}	\]
			and the vector-valued function $G:\mathbb{R} \rightarrow \mathbb{R}^2$ by $G(x)=(x, x^2)$ for all $x \in \mathbb{R}$.\\
			It is clear that $G$ is G\^{a}teaux differentiable function at $\bar{x}=0$ in every direction. Note that $\bar{y}=G(\bar{x})=(0, 0)$ and for any $h \in \mathbb{R}^2$, we have
			\begin{eqnarray*}
				\lim_{\substack{%
						\lambda \to 0+\\
					}}\frac{1}{\lambda}\odot\left(\textbf{F}(\bar{y}+\lambda h)\ominus_{gH}\textbf{F}(\bar{y})\right)
					&=&	\lim_{\substack{%
							\lambda \to 0+\\
						}}\frac{1}{\lambda}\odot\left(	\left( \frac{\lambda^6h_1^6}{(\lambda h_2-\lambda^2 h_1^2)^2+\lambda ^8 h_1^8}\right) \odot [2,~6] \right) = \textbf{0}.
					\end{eqnarray*}
					Then, due to the linearity and $gH$-continuity of the limit value, $\textbf{F}$ is also $gH$-G\^{a}teaux differentiable IVF at $ \bar{y}=G(\bar{x})$.\\
					The composition of $\textbf{F}$ and $G$ is
					\[\textbf{H}(x)=(\textbf{F}\circ G)(x)=
					\begin{cases}
					\left( \frac{1}{x^2}\right) \odot [2,~6], & \text{if } (x,y) \neq (0, 0),\\
					\textbf{0}, & \text{otherwise}.
					\end{cases}	\]
					Since for $h\neq0$,
					\[\lim_{\lambda \to 0+}\frac{1}{\lambda}\odot\left(\textbf{H}(\bar{x}+\lambda h)\ominus_{gH}\textbf{H}(\bar{x})\right)=\lim_{\lambda \to 0+}\frac{1}{\lambda^3h}\odot[2, 6]\]
					does not exist, $\textbf{H}=\textbf{F}\circ G$ is not $gH$-directional differentiable IVF at $G(\bar{x})=0$ in any direction $h\neq 0$.
					
				\end{example}	
				
				\begin{theorem} \label{th33dd}
					Let $I$ be a finite set of indices and $\textbf{F}_i: \mathcal{X} \rightarrow I(\mathbb{R})$ be a family of IVFs such that ${\textbf{F}} _{i _\mathscr{H}}(\bar{x})(h)$\\ exists for all $h \in \mathcal{X}.$ For each $x\in \mathcal{X}$, let the intervals in $\{\textbf{F}_i(x): i \in I\}$ be comparable. If $\textbf{F}(x) = \underset{i \in I}{\max}~ \textbf{F}_i(x)$ for all $x \in \mathcal{X}$, then,
					\[
					\textbf{F}_{\mathscr{H}}(\bar{x})(h) = \underset{i \in \mathcal{A}(\bar{x})}{\max}~ {\textbf{F}} _{i _\mathscr{H}}(\bar{x})(h), \text{ where } \mathcal{A}(\bar{x}) = \{ i \in I: \textbf{F}_{i}(\bar{x}) = \textbf{F}(\bar{x})\}.
					\]
				\end{theorem}
				
				\begin{proof}
					Let $\bar{x}\in\mathcal{X}$ and $d \in\mathcal{X}$ be such that $\bar{x}+ \lambda d \in\mathcal{X}$ for $\lambda >$ 0. Then,
					\begin{align}\label{00dd}
						&\textbf{F} _{i}(\bar{x} + \delta d)~\preceq~\textbf{F}(\bar{x} + \delta d), \quad \text{for all }i \in I \nonumber\\
						\text{or, }& \textbf{F} _{i}(\bar{x} + \delta d)\ominus _{gH} \textbf{F} (\bar{x})~\preceq~\textbf{F}(\bar{x} + \delta d) \ominus _{gH} \textbf{F} (\bar{x}),
						\quad \text{for all }i \in I \nonumber\\
						\text{or, }&\textbf{F} _{i}(\bar{x} + \delta d) \ominus _{gH} \textbf{F} _{i} (\bar{x})~\preceq~\textbf{F}(\bar{x} + \delta d) \ominus _{gH} \textbf{F} (\bar{x}), \quad \text{for each } i\in \mathcal{A}(\bar{x}) \nonumber\\
						\text{or, }&\lim_{\substack{%
								\delta \to 0+\\
								d \to h}} \frac{1}{\delta} \odot (\textbf{F} _{i}(\bar{x} + \delta d) \ominus _{gH} \textbf{F} _{i}(\bar{x}))~\preceq~\lim_{\substack{%
								\delta \to 0+\\
								d \to h}} \frac{1}{\delta} \odot
						(\textbf{F} (\bar{x} + \delta d) \ominus _{gH} \textbf{F} (\bar{x})) \nonumber\\
						\text{or, }&\underset{i \in \mathcal{A}(\bar{x})}{\max}~ {\textbf{F}} _{i _\mathscr{H}}(\bar{x})(h)~\preceq~\textbf{F}_{\mathscr{H}}(\bar{x})(h).~
					\end{align}
					To prove the reverse inequality, we claim that there exists a neighbourhood $\mathcal{N}(\bar{x})$ such that $\mathcal{A}(x) \subset \mathcal{A}(\bar{x}) \text{ for all } x \in \mathcal{N}(\bar{x})$. Assume on contrary that there exists a sequence $\{x_{k}\}$ in $\mathcal{X}$ with $x_{k}\to \bar{x}$ such that $\mathcal{A}(x_{k}) \not\subset \mathcal{A}(\bar{x})$.
					We can choose $i_{k} \in \mathcal{A}(x_{k})$ but $i_{k} \notin \mathcal{A}(\bar{x})$. Since $\mathcal{A}(x_{k})$ is closed, $i_{k} \to \bar{i}\in \mathcal{A}(x_{k})$. By $gH$-continuity of $\textbf{F}$, we have
					\[\textbf{F} _{\bar{i}}(x_{k}) = \textbf{F} (x_{k}) \implies \textbf{F} _{\bar{i}}(\bar{x}) = \textbf{F}(\bar{x}),\]
					which contradicts to $i_{k} \notin \mathcal{A}(\bar{x})$. Thus, $\mathcal{A}(x) \subset \mathcal{A}(\bar{x}) \text{ for all } x \in \mathcal{N}(\bar{x})$.\\
					Let us choose a sequence $\{\delta _{k}\},\delta_{k} \to 0$ such that $\bar{x} + \delta _{k} d \in \mathcal{N}(\bar{x})$ for all $d\in \mathcal{X}$. Then,
					\begin{align} \label{0_2dd}
						&  \textbf{F} _{i}(\bar{x})~\preceq~\textbf{F}(\bar{x}), \quad \text{for all } i \in I\nonumber \\
						\text{or, }& \textbf{F} (\bar{x} + \delta _{k} d) \ominus _{gH} \textbf{F} (\bar{x}) \preceq  \textbf{F} (\bar{x} + \delta _{k} d) \ominus _{gH} \textbf{F} _{i}(\bar{x}), 
						\quad \text{for all } i \in \mathcal{A}(\bar{x})\nonumber  \\
						\text{or, }& \textbf{F} (\bar{x} + \delta_{k} d) \ominus _{gH} \textbf{F} (\bar{x})  \preceq \textbf{F}_{i}(\bar{x} + \delta_{k} d) \ominus _{gH} \textbf{F}_{i}(\bar{x}), \quad \text{for all } i\in \mathcal{A}(\bar{x}+\delta_k d) \nonumber \\
						\text{or, }& 	\lim_{\substack{%
								k \to \infty\\
								d \to h}} \frac{1}{\delta_{k}} \odot (\textbf{F} (\bar{x} + \delta_{k} d) \ominus _{gH} \textbf{F} (\bar{x})) \preceq \lim_{\substack{%
								k \to \infty\\
								d \to h}}\frac{1}{\delta_{k}} \odot (\textbf{F}_{i}(\bar{x} + \delta_{k} d) \ominus _{gH} \textbf{F}_{i}(\bar{x})) \nonumber \\
						\text{or, }& \textbf{F}_{\mathscr{H}}(\bar{x})(h) \preceq  \underset{i \in \mathcal{A}(\bar{x})}{\max}~ {\textbf{F}} _{i _\mathscr{H}}(\bar{x})(h).
					\end{align}
					From (\ref{00dd}) and (\ref{0_2dd}) , we obtain
					\[
					\textbf{F} _{\mathscr{H}}(\bar{x})(h) = \max \textbf{F} _{i_\mathscr{H}}(\bar{x})(h) \text{ for all } i\in \mathcal{A}(\bar{x}).	\]
				\end{proof}


				\section{Characterization of Efficient Solutions}\label{section5}
				
				In this section, we present some characterizations of efficient solutions for IOPs with the help of the properties of $gH$-Hadamard differentiable IVFs.
				
				\begin{theorem}\label{tgd}$($\emph{Sufficient condition for efficient points}$)$.
					Let $\mathcal{S}$ be a nonempty convex subset of $\mathcal{X}$ and $\textbf{F}: \mathcal{S} \rightarrow I(\mathbb{R})$ be a convex IVF. If the function $\textbf{F}$ has a $gH$-Hadamard derivative at $\bar{x} \in \mathcal{S}$ in the  direction $v-\bar{x}$ with
					\begin{equation}\label{cep1}
						\textbf{F}_\mathscr{H}(\bar{x})(v-\bar{x})\nprec \textbf{0}, \quad \text{for all } v \in \mathcal{X},
					\end{equation}
					then $\bar{x}$ must be an efficient point of the IOP (\ref{IOP}).
				\end{theorem}
				
				\begin{proof}
					Assume that $\bar{x}$ is not an efficient point of $\textbf{F}$. Then, there exists at least one $y\in \mathcal{S}$ such that for any $\lambda\in (0,1]$, we have
					\begin{align*}
						& \lambda\odot\textbf{F}(y)\prec\lambda\odot\textbf{F}(\bar{x}), \\
						\text{or, } & \lambda\odot\textbf{F}(y)\oplus\lambda'\odot\textbf{F}(\bar{x})\prec\lambda\odot\textbf{F}(\bar{x})\oplus\lambda'\odot\textbf{F}(\bar{x}),~\text{where}~\lambda'=1-\lambda, \\
						\text{or, }& \lambda\odot\textbf{F}(y)\oplus\lambda'\odot\textbf{F}(\bar{x})\prec(\lambda+\lambda')\odot\textbf{F}(\bar{x})=\textbf{F}(\bar{x}).
					\end{align*}
					Due to the convexity of $\textbf{F}$ on $\mathcal{S}$, we have
					\begin{align}\label{k}
						& \textbf{F}(\bar{x}+\lambda (y-\bar{x}))=\textbf{F}(\lambda y+\lambda'\bar{x})\preceq\lambda\odot\textbf{F}(y)\oplus\lambda'\odot\textbf{F}(\bar{x})\prec\textbf{F}(\bar{x}), \nonumber\\
						\text{or, } & \textbf{F}(\bar{x}+\lambda (y-\bar{x}))\ominus_{gH}\textbf{F}(\bar{x})\prec \textbf{0}, \nonumber \\
						\text{or, } & \textbf{F}_\mathscr{H}(\bar{x})(v-\bar{x})\preceq \textbf{0}.
					\end{align}
					Now we have the following two possibilities.
					\begin{enumerate}[$\bullet$ \textbf{Case} I:]
						\item If $\textbf{F}_\mathscr{H}(\bar{x})(v-\bar{x})= \textbf{0}$, then $\textbf{F}_\mathscr{D}(\bar{x})(v-\bar{x})= \textbf{0}$ and
						\begin{equation}\label{ci}
							\underline{f}_\mathscr{D}(\bar{x})(v-\bar{x})=0 \text{ and } \overline{f}_\mathscr{D}(\bar{x})(v-\bar{x})=0.
						\end{equation}
						Due to Lemma \ref{lc1}, $\underline{f}$ and $\overline{f}$ are convex on $\mathcal{S}$. From (\ref{ci}), we observe that $\bar{x}$ is a minimum point of $\underline{f}$ and $\overline{f}$. Consequently, $\bar{x}$ is an efficient point of $\textbf{F}$. This contradicts to our assumption that $\bar{x}$ is not efficient point of $\textbf{F}$.
						
						\item If $\textbf{F}_\mathscr{H}(\bar{x})(v-\bar{x}) \prec \textbf{0}$, then this contradicts the assumption that $\textbf{F}_\mathscr{H}(\bar{x})(v-\bar{x})\nprec \textbf{0}$ for all $v\in \mathcal{X}$.
					\end{enumerate}
					Hence, $\bar{x}$ is the efficient point of the IOP (\ref{IOP}).
				\end{proof}
				
				\begin{remark}\label{ne1}
					The relation (\ref{cep1}) can be seen as a variational inequality for interval-valued functions. For details as variational inequalities, we refer \cite{ansari2013generalized}. The converse of Theorem \ref{tgd} is not true. For example, consider  $\mathcal{X}= \mathbb{R}$, $\mathcal{S}=[-1, 2]$, and the convex IVF $\textbf{F}: \mathcal{S} \rightarrow I(\mathbb{R})$ defined by
					\[\textbf{F}(x)=[4x^2-4x+1, 2x^2+75].\]
					At $\bar{x}=0$ and for $v\in \mathcal{X}$, $\textbf{F}_\mathscr{H}(\bar{x})(v)=v\odot [-4, 0]$ for all $v\in \mathcal{X}$.\\
					From Figure \ref{fexcd1}, it is clear that $\bar{x}=0$ is an efficient solution of the IOP (\ref{IOP}).  However, for all $v>0$ we have $\textbf{F}_\mathscr{H}(\bar{x})(v) \prec \textbf{0}.$
					\begin{figure}[H]
						\begin{center}
							\includegraphics[scale=0.7]{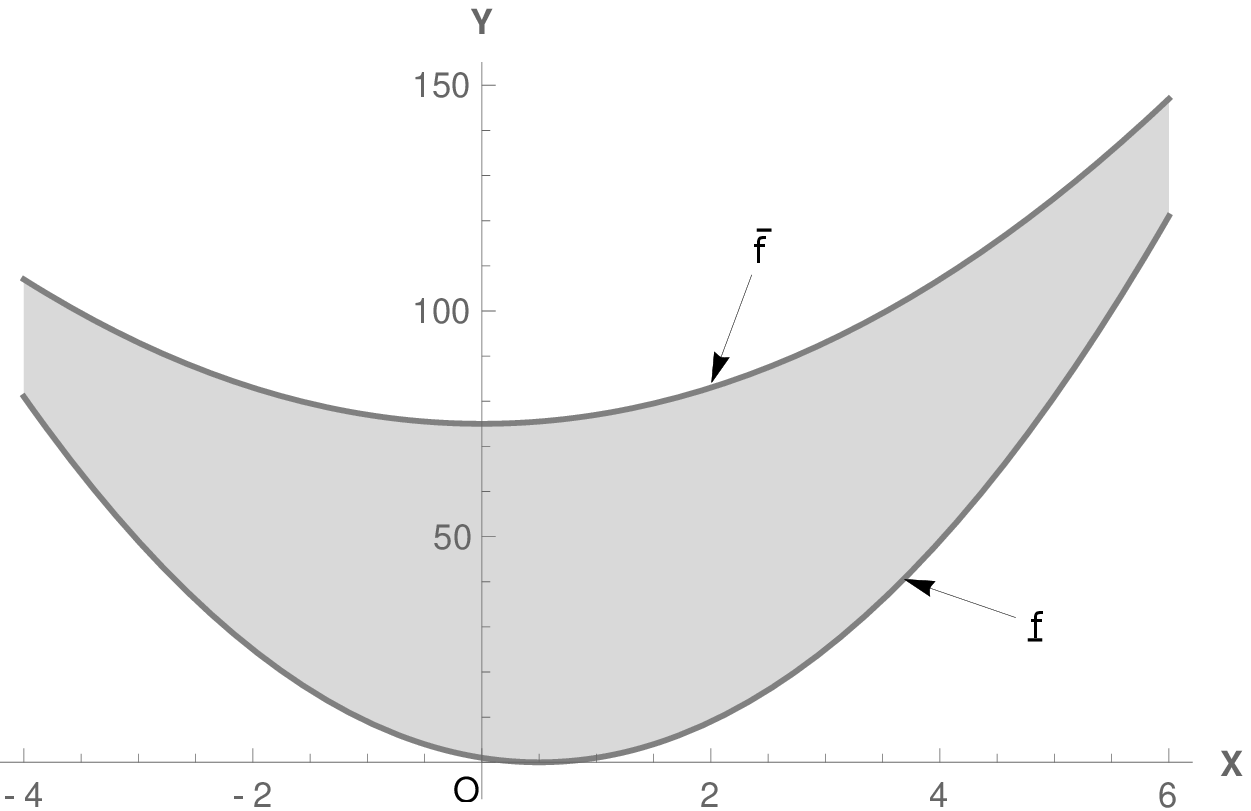}
							\caption{IVF $\textbf{F}$ of Remark \ref{ne1}}\label{fexcd1}
						\end{center}
					\end{figure}
					
				\end{remark}

				\begin{theorem}\label{eeth42} $($\emph{Necessary condition for efficient points}$)$.
					Let $\mathcal{S}$ be a linear subspace of $\mathcal{X}$, $\textbf{F}: \mathcal{S} \rightarrow I(\mathbb{R})$ be an IVF and $\bar{x}\in \mathcal{S}$ be an efficient point of the IOP (\ref{IOP}). If the function $\textbf{F}$ has a $gH$-Hadamard derivative at $\bar{x}$ in every direction $v\in \mathcal{S}$, then \[
					\textbf{F}_\mathscr{H}(\bar{x})(v-\bar{x})\nprec \textbf{0}, \quad \text{for all  } v\in \mathcal{S}.
					\]
				\end{theorem}
				\begin{proof}
					Since the point $\bar{x}$ is an efficient point of the function $\textbf{F}$, for any $h \in \mathcal{S}$ and $\lambda > 0$, we have
					\begin{align}\label{algn1}
						& \textbf{F}(\bar{x}+\lambda (h-\bar{x}))\ominus_{gH}\textbf{F}(\bar{x})\nprec \textbf{0}.
					\end{align}
					If $\textbf{F}_\mathscr{H}(\bar{x})(v-\bar{x}) \preceq \textbf{0}$, then due to linearity of $\textbf{F}_\mathscr{H}(\bar{x})$ on $\mathcal{S}$, we have $\textbf{F}_\mathscr{H}(\bar{x})(v-\bar{x}) = \textbf{0}$ by (\ref{nb2}) of Lemma \ref{nb}. Therefore, $\textbf{F}_\mathscr{H}(\bar{x})(v-\bar{x})\nprec \textbf{0} \text{ for all  } v\in \mathcal{S}.$\\
					If $\textbf{F}_\mathscr{H}(\bar{x})(v-\bar{x}) \nprec \textbf{0}$, then the result holds.
				\end{proof}

				\begin{remark}\label{in}
					One may think that in Theorem \ref{eeth42}, instead of considering the fact that the IVF $\textbf{F}$ is defined on a linear subspace of $\mathcal{S}$, we may take $\textbf{F}$ being defined on any nonempty convex subset of $\mathcal{S}$. However, this assumption is not sufficient. For instance, consider $\mathcal{X} = \mathbb{R}$, $\mathcal{S}=[-1, 7]$, and the convex IVF $\textbf{F}: \mathcal{S} \rightarrow I(\mathbb{R})$ defined by $\textbf{F}(x)= [x^2-4x+4, x^2+5].$ Then at $\bar{x}\in \mathcal{S}$, $\textbf{F}_\mathscr{H}(\bar{x})(v)=2v\odot [\bar{x}-2, \bar{x}]$ for all $v \in \mathcal{X}$. Note that $\bar{x}=0$ is an efficient point of IOP (\ref{IOP}) because $\textbf{F}(y) \nprec \textbf{F}(\bar{x}) \text{ for all } y \in \mathcal{S}$.
					However, $ \textbf{F}_\mathscr{H}(\bar{x})(v) \prec \textbf{0} $ for all $v>0$.
				\end{remark}

				\begin{theorem}\label{eeth42r}
					Let $\mathcal{S}$ be a nonempty subset of $\mathcal{X}$, $\textbf{F}: \mathcal{S} \rightarrow I(\mathbb{R})$ be an IVF, and $\bar{x}\in \mathcal{S}$ be an efficient point of the IOP (\ref{IOP}). If the IVF $\textbf{F}$ has a $gH$-Hadamard derivative at $\bar{x}$ in every direction $v\in \mathcal{S}$, then there exist no $ v\in \mathcal{S}$ such that $\textbf{F}_\mathscr{H}(\bar{x})(v-\bar{x}) < \textbf{0}.$
				\end{theorem}
				
				\begin{proof}
					Since the point $\bar{x}$ is an efficient point of the function $\textbf{F}$, for any $h \in \mathcal{S}$ and $\lambda > 0$, we have
					\begin{align*}\label{algn1r}
						& \textbf{F}(\bar{x}+\lambda (h-\bar{x}))\ominus_{gH}\textbf{F}(\bar{x})\nprec \textbf{0}.
					\end{align*}
					This implies that
					\begin{equation}\label{r}
						\lim_{\lambda \to 0+} \frac{1}{\lambda} \max\{\underline{f}(\bar{x}+\lambda (h-\bar{x}))- \underline{f}(\bar{x}), \overline{f}(\bar{x}+\lambda (h-\bar{x}))- \overline{f}(\bar{x}) \} \geq 0.
					\end{equation}
					From (\ref{r}) and Lemma \ref{rr}, there is no $ v\in \mathcal{S}$ such that $\textbf{F}_\mathscr{H}(\bar{x})(v-\bar{x}) < \textbf{0}.$
				\end{proof}
				
				\begin{theorem}\label{eeth42ddd}.
					Let $\mathcal{S}$ be a linear subspace of $\mathcal{X}$, $\textbf{F}: \mathcal{S} \rightarrow I(\mathbb{R})$ be an IVF, and $\bar{x}\in \mathcal{S}$ be an efficient point of the IOP (\ref{IOP}). If the IVF $\textbf{F}$ has a $gH$-Hadamard derivative at $\bar{x}$ in every direction $v\in \mathcal{S}$, then \[0 \in \textbf{F}_\mathscr{H}(\bar{x})(v), \quad \text{for all } v \in \mathcal{S}.\]
					The converse holds if $\textbf{F}$ is convex on $\mathcal{X}$.
				\end{theorem}
				
				\begin{proof}
					Let $\bar{x}$ be an efficient point of IOP (\ref{IOP}). Then, by Theorem \ref{eeth42}, we have $\textbf{F}_\mathscr{H}(\bar{x})(v)\nprec \textbf{0} \text{ for all  } v\in \mathcal{S}.$ Due to linearity of $\textbf{F}_\mathscr{H}(\bar{x})$ and $v=-h$, we obtain $\textbf{F}_\mathscr{H}(\bar{x})(h)\nsucc \textbf{0} \text{ for all  } h\in \mathcal{S}$. Hence, $0 \in \textbf{F}_\mathscr{H}(\bar{x})(v) \text{ for all  }$\\$ v\in \mathcal{S}.$
					
					Conversely, let $\textbf{F}$ be convex on $\mathcal{S}$ and assume that \textbf{F} has a $gH$-Hadamard derivative at $\bar{x}$ in every direction $w\in \mathcal{X}$. Let $\text{0} \in \textbf{F}_\mathscr{H}(\bar{x})(w) $ for all $w\in \mathcal{X}$. Then, due to linearity of $\textbf{F}_\mathscr{H}(\bar{x})$ on $\mathcal{S}$, we have		\[\textbf{F}_\mathscr{H}(\bar{x})(w) \nprec \textbf{0}~~\text{and}~~\textbf{0}\nprec\textbf{F}_\mathscr{H}(\bar{x})(w)~~\text{for all}~~w.\]
					Hence, $\bar{x}$ is efficient point of IOP (\ref{IOP}) by Theorem \ref{tgd}.
				\end{proof}
				
				
				\section{\textbf{Fritz John and Karush-Kuhn-Tucker Optimality Conditions}} \label{section5d}
				
				In this section, we derive an extended KKT necessary and sufficient optimality conditions to characterize efficient solutions of IOPs.
				\begin{lemma}\label{K1} Let $\textbf{F}:\mathbb{R}^n \to I(\mathbb{R})$ be a $gH$-Hadamard differentiable IVF at $\bar{x}$ in the direction $v\in \mathbb{R}^n$ with $\textbf{F}_\mathscr{H}(\bar{x})(v) \prec \textbf{0}$. Then, there exists $\delta >0$ such that for each $\lambda \in (0, \delta)$,
					\[\textbf{F}(\bar{x}+\lambda v) \prec \textbf{F}(\bar{x}).\]
				\end{lemma}
				
				\begin{proof}
					Since $\textbf{F}_\mathscr{H}(\bar{x})(v) \prec \textbf{0}$, there exist $\delta, \delta' >0$ such that for all $h \in \mathbb{R}^n$, we have
					\[\frac{1}{\lambda}\odot(\textbf{F}(\bar{x}+\lambda h)\ominus_{gH}\textbf{F}(\bar{x})) \prec \textbf{0},~\lambda \in (0, \delta)~\text{and}~\lVert v-h \rVert <\delta'.\]
					Due to $gH$-continuity of $\textbf{F}$ at $v$, we get
					\begin{eqnarray*}
						&&\textbf{F}(\bar{x}+\lambda v)\ominus_{gH}\textbf{F}(\bar{x}) \prec \textbf{0},~\forall~\lambda \in (0, \delta),
					\end{eqnarray*}
					which implies $\textbf{F}(\bar{x}+\lambda v)\prec\textbf{F}(\bar{x}),~\forall~\lambda \in (0, \delta).$
				\end{proof}
				
				\begin{definition}\label{dK2} Let $\textbf{F}:\mathbb{R}^n \to I(\mathbb{R})$ be a $gH$-Hadamard differentiable IVF at $\bar{x}$. Then, the set of descent directions at $\bar{x}$ is defined by
					\[\hat{\textbf{F}}(\bar{x})= \{d \in \mathbb{R}^n: \textbf{F}_\mathscr{H}(\bar{x})(d)\prec \textbf{0} \}.\]
					As for any  $d$ in $\hat{\textbf{F}}(\bar{x})$, $\lambda d \in \hat{\textbf{F}}(\bar{x})$ for all $\lambda >0$, the set $\hat{\textbf{F}}(\bar{x})$ is called the cone of descent direction.
					
				\end{definition}
				
				\begin{definition}\label{dK3} \cite{Ghosh2020} Given a nonempty set $\mathcal{S}\subseteq \mathbb{R}^n$ and $\bar{x} \in \mathcal{S}$. At $\bar{x}$, the cone of feasible directions of $\mathcal{S}$ is defined by
					\[\hat{\mathcal{S}}(\bar{x})= \{d \in \mathbb{R}^n: d \neq 0,~ \bar{x}+\lambda d \in \mathcal{S}, ~\forall~ \lambda \in (0, \delta)~\text{and for some}~\delta >0\}.\]
					
				\end{definition}
				
				\begin{lemma}\label{lk4}
					Let $\mathcal{S}\subseteq \mathbb{R}^n$ 
					and $\textbf{F}:\mathbb{R}^n \to I(\mathbb{R})$ be a $gH$-Hadamard differentiable IVF at $\bar{x} \in \mathcal{S}$. If $\bar{x}$ is an efficient solution of the IOP (\ref{IOP}), then $\hat{\textbf{F}}(\bar{x}) \cap \hat{\mathcal{S}}(\bar{x})= \emptyset.$
				\end{lemma}
				
				\begin{proof}
					Assume contrary that $\hat{\textbf{F}}(\bar{x}) \cap \hat{\mathcal{S}}(\bar{x}) \neq \emptyset$ and $d \in \hat{\textbf{F}}(\bar{x}) \cap \hat{\mathcal{S}}(\bar{x})$. By Lemma \ref{K1} and Definition \ref{dK3}, there exist $\delta_1 , \delta_2>0$ such that
					\[\bar{x}+\lambda d \in \mathcal{S}~\text{for all}~\lambda ~\text{in}~ (0, \delta_1) \quad \text{and} \quad \textbf{F}(\bar{x}+\lambda d) \prec \textbf{F}(\bar{x})~\text{for all}~\lambda ~\text{in}~(0, \delta_2).\]
					Taking $\delta = \min\{\delta_1, \delta_2\}$, we see that for all $\lambda \in (0, \delta)$,
					\[\bar{x}+\lambda d \in \mathcal{S}~\text{and}~\textbf{F}(\bar{x}+\lambda d) \prec \textbf{F}(\bar{x}).\]
					This is contradictory to $\bar{x}$ being a local efficient point. Hence, $\hat{\textbf{F}}(\bar{x}) \cap \hat{\mathcal{S}}(\bar{x})= \emptyset.$
				\end{proof}

				\begin{lemma}{\label{prop1}}
					For $i=1,2,\ldots,m$, let $\textbf{G}_{i}: \mathbb{R}^n \rightarrow I(\mathbb{R})$ be IVF, $X$ be a non-empty open set in $\mathbb{R}^n$, and $\mathcal{S}$ = $\{x \in X: \textbf{G}_{i}(x) \preceq  \mathbf{0} \text{ for }i= 1,2,\dotsc,m \}$. Let $\bar{x} \in \mathcal{S}$ and $\textit{I}(\bar{x}) = \{i: \textbf{G}_{i}(\bar{x}) = \mathbf{0} \}$. For all $i \in I(\bar{x})$, assume that $\textbf{G}_{i}$ is $gH$-Hadamard differentiable at $\bar{x}$ and $gH$-continuous for $i \notin I(\bar{x})$, define
					\[ \hat{G}(\bar{x}) = \{d : {\textbf{G}_i}_\mathscr{H}(\bar{x})(d)_(\bar{x})  \prec \mathbf{0} \ \text{for  all } i \in I(x_0) \} .\]
					Then,
					$ \hat{G}(\bar{x}) \subseteq \hat{\mathcal{S}}(\bar{x}),$
					where $\hat{\mathcal{S}}(\bar{x}) = \{d \in \mathbb{R}^n :d \neq 0,~ \bar{x} + \alpha d \in S ~ \forall \alpha~ \in (0,\delta)  \text{ for some } \delta >0 \}$.
				\end{lemma}
				
				\begin{proof}
					It is similar to proof of Lemma 3.1 in \cite{Ghosh2020} for $gH$-Hadamard derivative, and therefore, we omit.
				\end{proof}

				With the help of Lemma \ref{prop1}, we characterize an efficient solution of a constrained IOP. It is shown that at a local efficient solution, the cones of descent direction and feasible direction have an empty intersection.

				\begin{theorem}{\label{theorem13}}
					Let $\mathcal{S}$ be a non-empty open set in $\mathbb{R}^n$. Consider an IOP
					\begin{equation}\label{const_IOP}
						\left.
						\begin{aligned}
							&\min ~ {\textbf F}(x) \\
							&~\textup{such that }~ \textbf{G}_{i}(x) \preceq \mathbf{0},~\textup{ for $i = 1,2,\dotsc,m$} \\
							&~ x \in \mathcal{S},
						\end{aligned}
						\right\}
					\end{equation}
					where ${\textbf F} : \mathbb{R}^n \rightarrow I(\mathbb{R})$ and $\textbf{G}_{i}: \mathbb{R}^n \rightarrow I(\mathbb{R})$ for $i=1,2,\dotsc,m$.
					For a feasible point $x_0$, define $ \textit{I}(x_0) = \{i:   \textbf{G}_{i}(\bar{x})=0 \}$. At $\bar{x}$, let ${\textbf F}$ and  $\textbf{G}_{i}$, $i \in I(\bar{x})$, be $gH$-Hadamard differentiable, and for $i \notin I(\bar{x})$, $\textbf{G}_{i}$ be $gH$-continuous.
					If $\bar{x}$ is a local efficient solution of (\ref{const_IOP}), then
					\[\hat{F}(\bar{x}) \cap \hat{G}(\bar{x}) = \emptyset,\] where $ \hat{F}(\bar{x}) = \{d : \textbf{F}_\mathscr{H}(\bar{x})(d)_(\bar{x})  \prec \mathbf{0} \} \ \textup{ and } \hat{G}(\bar{x}) = \{d : {\textbf{G}_i}_\mathscr{H}(\bar{x})(d)_(\bar{x})  \prec \mathbf{0} \text{ for each } i \in I(\bar{x}) \}$.
				\end{theorem}

				\begin{proof}
					By Lemma \ref{lk4} and Lemma \ref{prop1}, we obtain
					\[ \text{$x_0$ is a local efficient solution } \implies \hat{F}(\bar{x}) \cap \hat{\mathcal{S}}(\bar{x}) = \emptyset \implies \hat{F}(\bar{x}) \cap \hat{G}(\bar{x}) = \emptyset.\]
				\end{proof}

				
				

				\begin{theorem}{\label{theorem16}}
					\textup{(Extended Fritz John necessary optimality condition).}
					Let $\mathcal{S}$ be a non-empty open set in $\mathbb{R}^n$; ${\textbf F} : \mathbb{R}^n \rightarrow I(\mathbb{R})$ and $\textbf{G}_{i}: \mathbb{R}^n \rightarrow I(\mathbb{R})$ for $i = 1,2,\dotsc,m$ be IVFs. Consider the IOP:
					\begin{equation}\label{CIOP_DG}
						\left.
						\begin{aligned}
							&~\min~ {\textbf F}(x),\\
							&~\textup{such that} ~ \textbf{G}_{i}(x) \preceq  \mathbf{0}, \quad i = 1,2,\dotsc,m \\
							~&~ x \in \mathcal{S}.
						\end{aligned}
						\right\}
					\end{equation}
					For a feasible point $\bar{x}$, define $\textit{I}(\bar{x}) = \{i:  {\textbf{G}_{i}}(\bar{x}) = \textbf{0} \}$. Let ${\textbf F}$ and $\textbf{G}_{i}$  be $gH$-Hadamard differentiable at $\bar{x}$ for $i \in I(\bar{x})$ and $gH$-continuous for $i \notin I(\bar{x})$.
					If $\bar{x}$ is a local efficient point of (\ref{CIOP_DG}), then there exist constants $u_0 \ and \ u_i$ for $i \in I(\bar{x})$ such that
					\begin{equation*}
						\left\{
						\begin{aligned}
							& 0 \in \left (u_0 \odot \textbf{F}_\mathscr{H}(\bar{x})(d) \oplus \sum_{i \in I(\bar{x})} u_i \odot {\textbf{G}_i}_\mathscr{H}(\bar{x})(d)\right), \\
							& u_0 \ge 0, u_i \geq 0 \text{ for } i \in I(\bar{x}), \\
							& \left( u_0, {u_I}  \right) \neq \left(0,  {0}_v^{|I(\bar{x})|} \right),
						\end{aligned}
						\right.
					\end{equation*}
					where ${u_I}$ is the vector whose components are $u_i$ for $i \in I(\bar{x})$. \\
					Further, if $\textbf{G}_{i}$, for all $i \notin I(\bar{x})$, are also $gH$-Hadamard differentiable at $\bar{x}$, then there exist constants $u_0, u_1, u_2, \ldots, u_m$ such that
					\begin{equation*}
						\left\{
						\begin{aligned}
							& 0 \in \left (u_0 \odot \textbf{F}_\mathscr{H}(\bar{x})(d) \oplus \sum_{i=1}^{m} u_i\odot {\textbf{G}_i}_\mathscr{H}(\bar{x})(d)\right), \\
							& u_i \odot {\textbf{G}_{i}(\bar{x})} = \textbf{0},~ i = 1, 2, \dotsc, m, \\
							& u_0 \geq 0, u_i \geq 0,~  i = 1,2,\dotsc,m, \\
							& \left( u_0, {u}  \right) \neq \left(0,  {0}_v^m \right),
						\end{aligned}
						\right.
					\end{equation*}
					where ${u}$ is the vector $(u_1, u_2, \ldots, u_m)$.
				\end{theorem}

				\begin{proof}
					Since $\bar{x}$ is a local efficient point of (\ref{CIOP_DG}), by Theorem \ref{theorem13}, we get
					\begin{align}
						& \hat{F}(\bar{x}) \cap \hat{G}(\bar{x}) = \emptyset, \nonumber \\
						\text{or, } & \nexists \ d \in \mathbb{R}^n \text{ s.t. } \textbf{F}_\mathscr{H}(\bar{x})(d) \prec \mathbf{0} \text{ and } {\textbf{G}_i}_\mathscr{H}(\bar{x})(d)  \prec \mathbf{0} ~\forall ~i~ \in I(\bar{x}), \nonumber\\
						\text{or, }&\textbf{F}_\mathscr{H}(\bar{x})(d)  \nprec \mathbf{0} \text{ and } {\textbf{G}_i}_\mathscr{H}(\bar{x})(d)  \nprec \mathbf{0} ~\forall ~d \in \mathbb{R}^n~\text{and}~ i \in I(\bar{x}), \nonumber\\
						\text{or, }& 0 \in \textbf{F}_\mathscr{H}(\bar{x})(d) \text{ and } 0 \in {\textbf{G}_i}_\mathscr{H}(\bar{x})(d) ~\forall ~d \in \mathbb{R}^n~\text{and}~ i~ \in I(\bar{x})~\text{ by Lemma \ref{nb}}.
					\end{align}
					We can chose nonzero vector $ p $ with $p=[u_0, u_i]_{i \in I(\bar{x})}^\top$
					such that
					\begin{equation*}
						\left\{
						\begin{aligned}
							& 0 \in \left (u_0 \odot \textbf{F}_\mathscr{H}(\bar{x})(d) \oplus \sum_{i \in I(\bar{x})} u_i \odot {\textbf{G}_i}_\mathscr{H}(\bar{x})(d)\right), \\
							& u_0, u_i \geq 0 \text{ for } i \in I(x_0), \\
							& \left( u_0, {u_I}  \right) \neq \left(0, 0, \cdots, 0 \right).
						\end{aligned}
						\right.
					\end{equation*}
					This proves the first part of the theorem. \\
					For $i \in I(\bar{x})$, ${\textbf{G}_{i}}(\bar{x}) = \textbf{0}$. Therefore, $ u_i \odot \textbf{G}_{i}(\bar{x}) = \textbf{0}$. If $\textbf{G}_{i}$ for all $i \notin I(\bar{x})$ are also $gH$-differentiable at $\bar{x}$, by setting $u_i = 0$ for $ i \notin I(\bar{x})$ the second part of the theorem is followed.
				\end{proof}
				
				\begin{definition}\label{linearindep}\cite{Ghosh2020} The set of $m$ intervals $\{ \textbf{X}_1, \textbf{X}_2, \ldots, \textbf{X}_m \}$ is said to be linearly independent if for $m$ real numbers $c_1$, $c_2$, \ldots, $c_m$:  \[0 \in c_1 \odot \textbf{X}_1 \oplus c_2 \odot \textbf{X}_2 \oplus  \ldots \oplus c_m \odot \textbf{X}_m \quad \text{if and only if}\quad c_1 = 0, c_2 = 0, \dotsc , c_m = 0. \]
				\end{definition}

				\begin{theorem}{\label{theorem17}}
					\textup{(Extended Karush-Kuhn-Tucker necessary optimality condition).}
					Let $\mathcal{S}$ be a non-empty open set in $\mathbb{R}^n$ and  ${\textbf F} : \mathbb{R}^n \rightarrow I(\mathbb{R})$ and $\textbf{G}_{i}: \mathbb{R}^n \rightarrow I(\mathbb{R})$, $i=1,2,\dotsc,m$, be IVFs. Suppose that $\bar{x}$ is a feasible point of the following IOP:
					\begin{equation*}
						\left.
						\begin{aligned}
							&~\min ~ {\textbf F}(x) \\
							&~\textup{such that } ~ \textbf{G}_{i}(x) \preceq  \mathbf{0},\quad i = 1,2,\dotsc,m \\
							~&~ x \in \mathcal{S}.
						\end{aligned}
						\right\}
					\end{equation*}
					Define $\textit{I}(\bar{x}) = \{i: {\textbf{G}_{i}}(\bar{x}) = 0\}$. Let
					\begin{enumerate}[(i)]
						\item  ${\textbf F}$ and  $\textbf{G}_{i}$ be $gH$-Hadamard differentiable at $\bar{x}$ for all $i \in I(\bar{x})$,
						\item $\textbf{G}_{i}$ be $gH$-continuous for all $i \notin I(\bar{x})$, and
						\item the collection of intervals $\{ {\textbf{G}_i}_\mathscr{H}(\bar{x})(d) ~:~ i \in I(\bar{x}) \}$ be linearly independent.
					\end{enumerate}
					If $\bar{x}$ is a local efficient solution, then there exist constants $u_i\geq 0$ for all $i \in I(\bar{x})$ such that
					\begin{equation*}
						0 \in \left (u_0 \odot \textbf{F}_\mathscr{H}(\bar{x})(d) \oplus \sum_{i \in I(\bar{x})} u_i\odot {\textbf{G}_i}_\mathscr{H}(\bar{x})(d)\right)
					\end{equation*}
					If $\textbf{G}_{i}$'s, for $i \notin I(\bar{x})$, are also $gH$-differentiable at $\bar{x}$, then there exist constants $u_1$, $u_2$, \ldots, $u_m$ such that
					\begin{equation*}
						\left\{
						\begin{aligned}
							& 0 \in \left (u_0 \odot \textbf{F}_\mathscr{H}(\bar{x})(d) \oplus \sum_{i=1}^{m}  u_i\odot {\textbf{G}_i}_\mathscr{H}(\bar{x})(d)\right) ,\\
							& u_i \odot \textbf{G}_{i}(\bar{x}) = \textbf{0},~ i = 1,2,\dotsc, m,\\
							& u_i \geq 0,~ i = 1,2,\dotsc, m .
						\end{aligned}
						\right.
					\end{equation*}
				\end{theorem}

				\begin{proof}
					By Theorem \ref{theorem16}, there exist real constants $u_0$ and $u'_i$ for all $i \in I(\bar{x})$, not all zeros, such that
					\begin{equation}
						\left\{
						\begin{aligned}
							& 0 \in \left (u_0 \odot \textbf{F}_\mathscr{H}(\bar{x})(d) \oplus \sum_{i \in I(\bar{x})} u'_i\odot {\textbf{G}_i}_\mathscr{H}(\bar{x})(d)\right), \\
							& u_0 \ge 0, u_i' \geq 0 \text{ for all }  i \in I(\bar{x}).
						\end{aligned}
						\right.
					\end{equation}
					Then, we must have $u_0 > 0$. Since otherwise, the set  $\{ {\textbf{G}_i}_\mathscr{H}(\bar{x})(d) ~:~ i \in I(\bar{x})\}$ will become linearly dependent. \\ \\
					Define $u_i = u_i' / u_0$. Then, $u_i \ge 0$ for all $i \in I(\bar{x})$ and
					\[0 \in \left (u_0 \odot \textbf{F}_\mathscr{H}(\bar{x})(d) \oplus \sum_{i \in I(\bar{x})} u_i\odot {\textbf{G}_i}_\mathscr{H}(\bar{x})(d)\right).\]
					For $i \in I(\bar{x})$, ${\textbf{G}_{i}}(\bar{x}) = \textbf{0}$. Therefore, $ 0 \in u_i \odot \textbf{G}_{i}(\bar{x})$. If the functions $\textbf{G}_{i}$ for $i \notin I(\bar{x})$ are also $gH$-Hadamard differentiable at $x_0$, then by setting $u_i = 0$ for $i \notin I(\bar{x})$,  the latter part of the theorem is followed.
				\end{proof}

				\begin{theorem}\label{kkts}$($\emph{Extended Karush-Kuhn-Tucker sufficient condition for efficient points}$)$. Let $\mathcal{S}$ be a nonempty convex subset of $\mathcal{X}$; $\textbf{F}: \mathcal{S} \to I(\mathbb{R})$ and $\textbf{G}_i:\mathcal{S} \to I(\mathbb{R}),~ i=1, 2, \cdots, m$ be interval-valued $gH$-Hadamard differentiable convex functions. Suppose that $\bar{x} \in \mathcal{S}$ is a feasible point of the following IOP:
					\begin{equation}
						\left.
						\begin{aligned}
							&~\text{ min }~ \textbf{F}(x) \\
							&~\text{ such that }~ \textbf{G}_i(\bar{x}) \preceq \textbf{0}, \quad i=1, 2, \cdots, m\\
							&~~x\in \mathcal{S}.
						\end{aligned}
						\right\}
					\end{equation}
					If there exist real constants $u_1, u_2, \ldots, u_m$ for which
					\[ \begin{cases}
					\textbf{F}_\mathscr{H}(\bar{x})(v) \oplus \sum_{i=1}^{m} u_i \odot {\textbf{G}_i}_\mathscr{H}(\bar{x})(v) \nprec \textbf{0}, \quad \text{ for all } v \in \mathcal{S},\\
					u_i \odot \textbf{G}_i(\bar{x})=\textbf{0}, ~i= 1, 2,\cdots, m\\
					u_i\geq 0,~i= 1, 2,\cdots, m,
					
					\end{cases}\]
					then $\bar{x}$ is an efficient point of the IOP.
					
				\end{theorem}
				
				\begin{proof}
					By the hypothesis, for every $v\in \mathcal{S}$ satisfying $\textbf{G}_i(v) \preceq \textbf{0}$ for all $i=1, 2,\ldots, m$, we have
					\begin{eqnarray*}
						&&	\textbf{F}_\mathscr{H}(\bar{x})(v-\bar{x}) \oplus \sum_{i=1}^{m} u_i {\textbf{G}_i}_\mathscr{H}(\bar{x})(v-\bar{x}) \nprec \textbf{0},\\
						&\implies& \left(\textbf{F}(v)\ominus_{gH} \textbf{F}(\bar{x})\right) \oplus \left( \sum_{i=1}^{m} u_i \left(\textbf{G}_i(v)\ominus_{gH} {\textbf{G}_i} (\bar{x})\right)\right) \nprec \textbf{0}\\
						&&(\text{ by (\ref{gc}) of Theorem \ref{eth33} and (\ref{31}) of Lemma \ref{forfrechet}}),\\
						&\implies& \left(\textbf{F}(v)\ominus_{gH} \textbf{F}(\bar{x})\right) \oplus \left( \sum_{i=1}^{m} u_i \left(\textbf{G}_i(v)\right)\right) \nprec \textbf{0},\\
						&\implies& \textbf{F}(v)\ominus_{gH} \textbf{F}(\bar{x})\nprec \textbf{0} \text{ from (\ref{32}) of Lemma \ref{forfrechet}},\\
						&\implies& \textbf{F}(v)\nprec \textbf{F}(\bar{x}).
					\end{eqnarray*}
					Hence, $\bar{x}$ is an efficient point of the IOP.
				\end{proof}
				
				
				\section{Application to Support Vector Machines}{\label{svm_application}}
				
				In many classification problems, the data set may not be precise and thus involves uncertainty. This may be due to errors in measurement, implementation, etc. For example, let us assume that we want to predict whether there will be rain tomorrow or not. The data we may require the wind speed, humidity levels, temperature, etc. These variables usually have values in intervals like 10--13 km/hr wind speed, 40--50\% humidity, $30-35^o C$ temperature, etc. The standard Support Vector Machines (SVM) formulation is not applicable for such data as these quantities are interval-valued. Thus, we formulate the SVM problem for the interval-valued data set \[\left\{ (\bm{X}_i, y_i) ~\middle|~ \bm{X}_i \in I(\mathbb{R})^n,~ y_i \in \{-1,1\}, i = 1, 2, \cdots, m \right\}\] by
				\begin{equation}{\label{svmadjust}}
					\left.
					\begin{aligned}
						&~\underset{w,b}{\min} ~F(w, b) = \tfrac{1}{2} \|w\|^2, \\
						&~\text{such that }~ \textbf{G}_{i}(w, b) = [1, 1] \ominus_{gH}y_i \odot \left(w^{\top} \odot \bm{X}_i \oplus b\right) \preceq \bm{0}, ~~i=1,2,\dotsc,m.
					\end{aligned}
					\right\}
				\end{equation}
				We note that the functions $F$ and $\textbf{G}_{i}$ are $gH$-Hadamard differentiable and convex. At $\bar{x}=(\bar{w}, \bar{b})$, in the direction $v=(w, b)$, we have
				\[\textbf{F}_\mathscr{H}(\bar{x})(v)  = w\odot [\bar{w}, \bar{w}]
				\text{ and } {\textbf{G}_i}_\mathscr{H}(\bar{x})(d)  = -\left(w\odot(y_i \odot \bm{X}_i)\oplus by_i\right).\]
				According to Theorem \ref{theorem17}, for an efficient point $(\bar{w}, \bar{b})$ of (\ref{svmadjust}), there exist nonnegative scalars $u_1, u_2, \ldots, u_m$ such that
				\begin{align}
					\label{kkt1} & \textbf{0} \in \left ( w\odot [\bar{w}, \bar{w}]  \oplus \sum_{i=1}^{m} u_i \odot -\left(w\odot(y_i \odot \bm{X}_i)\oplus by_i\right)\right), \\
					\text{ and }~ \label{kkt2}& \bm{0} = u_i \odot {\textbf{G}_{i}(w^*,b^*)}, \quad i = 1,2,\dotsc, m.
				\end{align}
				The condition (\ref{kkt1}) can be simplified as
				\begin{align*}
					\bm{0} \in \left( [w^*, w^*] \oplus \sum_{i=1}^{m} (-u_i y_i)\odot \bm{X}_i \right)
					\text{ and }~\sum_{i=1}^{m} u_i y_i =0.
				\end{align*}
				The data points $\bm{X}_i$ for which $u_i \neq 0$ are called support vectors. By (\ref{kkt2}), corresponding to any $u_i > 0$, we have $\textbf{G}_{i}(w^*,b^*) = \bm{0}$. Thus, corresponding to $w^*$, the value of the bias $b^*$ is such a quantity that $\textbf{G}_{i}(w^*,b^*) = \bm{0}$ for all of those $i \in \{1, 2, \ldots, m\}$ for which $u_i > 0$.
				
				As the functions $F$ and $\textbf{G}_{i}$ are $gH$-Hadamard differentiable and convex, by Theorems \ref{theorem17} and \ref{kkts}, the set of conditions by which we obtain the efficient solutions of the SVM IOP (\ref{svmadjust}) are
				\begin{equation}\label{final_equations}
					\left\{
					\begin{aligned}
						& \bm{0} \in \left( [w, w] \oplus \sum_{i=1}^{m} (-u_i y_i)\odot \bm{X}_i \right), \\
						& \sum_{i=1}^{m} u_i y_i = 0, \\
						& \bm{0} = u_i \odot \textbf{G}_{i}(w,b),~ i = 1,2,\dotsc, m.
					\end{aligned}
					\right.
				\end{equation}
				Corresponding to any of the value of $w$ that satisfies (\ref{final_equations}), we define the set of possible values of the bias by
				\begin{equation*}\label{bias_set}
					\bigcap_{i:~ u_i > 0} \left\{ b ~\middle|~ \textbf{G}_{i}(w,b) = \bm{0}\right\}.
				\end{equation*}
				By using any solution $\bar w$ and $\bar b$ of (\ref{final_equations}) and (\ref{bias_set}), a classifying hyperplane and the SVM classifier function are given by
				\[ \bar w^{\top}\bm{X} + \bar b = \bm{0} \quad \text{and} \quad s^*(\bm{X}) = \text{sign}\left( \bar w^{\top}\bm{X} + \bar b \right),~\text{where sign ($\cdot$) denotes the sign function. }  \]


				\section{Conclusion and Future Directions} \label{sect6}
				In this article, the concept of $gH$-hadamard derivative for IVFs has been studied (Definition \ref{derivative}). One can trivially notice that in the degenerate case, the Definition \ref{derivative} reduces to the respective conventional definition for the real-valued functions (see \cite{Yu2013,D2005}). It has been noticed that the $gH$-Hadamard derivative at any point is the $gH$-Fr\'{e}chet derivative at that point and vise-versa (Theorem \ref{th31}). Also, a $gH$-Hadamard differentiable IVF is found to be $gH$-continuous (Theorem \ref{the1}). It has been shown that the $gH$-Hadamard derivative is helpful to characterize the convexity of an IVF (Theorem \ref{eth33}). It is also observed that the composition of a Hadamard differentiable real-valued function and a $gH$-Hadamard differentiable IVF is $gH$-Hadamard differentiable IVF and the chain rule is applicable (Theorem \ref{th34edd} and Theorem \ref{th4444ed}). Further, for a finite number of IVFs whose values are comparable at each point, it has been proven that the $gH$-Hadamard derivative of maximum of all finite comparable IVFs is the maximum of their $gH$-Hadamard derivative (Theorem \ref{th33dd}).
				
				In addition, it has been shown that if the objective function of an IOP is convex on the feasible set  $\mathcal{S}\subseteq \mathcal{X}$ and $gH$-Hadamard derivative at $\bar{x} \in  \mathcal{S}$ does not dominate to $\textbf{0}$, then $\bar{x}$ is an efficient point of that IOP (Theorem \ref{tgd}). Further, it is proved that if the feasible set $\mathcal{S}$ is linear subspace of $\mathcal{X}$ and the objective function of IOP is $gH$-Hadamard differentiable at an efficient point of IOP, then $gH$-Hadamard derivative does not dominate to $\textbf{0}$ (Theorem \ref{eeth42}) and also contains $0$ (Theorem \ref{eeth42ddd}). Moreover, for constraint IOPs, we have proved extended KKT necessary and sufficient condition to characterize the efficient solutions by using $gH$-Hadamard derivative (Theorem \ref{theorem17} and Theorem \ref{kkts}).

				As an application of the proposed $gH$-Hadamard derivative, we have formulated and solved SVM problem  for interval-valued data.

				In analogy to the current study, future research can be carried out for other generalized directional derivatives for IVFs, e.g., upper and lower Dini semiderivative, Hadamard semiderivative, upper and lower Hadamard semiderivative, Michel-Penot, etc., and their relationships \cite{Delfour2012}.

				In parallel to the proposed analysis of IVFs, another promising direction of future research can be the analysis of the fuzzy-valued functions (FVFs) as the alpha-cuts of fuzzy numbers are compact intervals \cite{Ghoshfuzzy}. Hence, we expect that some results for FVFs can be obtained in a similar way to this paper.

				\appendix
				
				
				\section{Proof of Lemma \ref{forfrechet}} \label{nbe}
				\begin{proof}
					Let $\textbf{A} = [\underline{a}, \overline{a}]$ and $ \textbf{B} = [\underline{b}, \overline{b}]$.\\
					\begin{enumerate}[(i)]
						\item Since $\textbf{B}\nprec \textbf{0}$ and $\textbf{B}\preceq \textbf{A}$, then
						\[\overline{b}\geq 0~\text{and}~\overline{b}\geq \overline{a} \implies \overline{a}\geq 0 \implies \textbf{A}\nprec \textbf{0}.\]
						
						\item Since $\textbf{A} \oplus \textbf{B} \nprec \textbf{0}$ and $\textbf{B} \preceq \textbf{0}$, then
						\[\overline{a}+\overline{b}\geq 0 \text{ and } \overline{b}\leq 0 \implies \overline{a} \geq 0 \implies \textbf{A} \nprec \textbf{0}.\]
					\end{enumerate}
				\end{proof}
				
				\section{Proof of Lemma \ref{nb}}\label{anb}
				\begin{proof}
					\begin{enumerate}[(i)]
						\item If
						\begin{equation}\label{enb1}
							\textbf{F}(x) \nprec \textbf{0} \text{ for all } x \in \mathcal{S},
						\end{equation}
						then due to linearity of $\textbf{F}$, we have
						\begin{equation}\label{enb2}
							\textbf{F}(x)=(-1) \odot \textbf{F}(-x)  \nsucc \textbf{0} \text{ for all } x \in \mathcal{S}
						\end{equation}
						since $\textbf{F}(-x) \nprec \textbf{0}$ by (\ref{enb1}).
						From (\ref{enb1}) and (\ref{enb2}), it is clear that $\textbf{0}$ and $\textbf{F}(x)$ are not comparable.
						
						\item If
						$\textbf{F}(x) \preceq \textbf{0} \text{ for all } x \in \mathcal{S},$
						then due to linearity of $\textbf{F}$, we have
						$\textbf{F}(x)= (-1) \odot \textbf{F}(-x) \succeq \textbf{0} \text{ for all } x \in \mathcal{S}.$\\
						Hence, $\textbf{F}(x)= \textbf{0}$.
					\end{enumerate}
				\end{proof}

				\noindent
				\textbf{Acknowledgement}\\
				In this research, the first author (R. S. Chauhan) is supported by a research scholarship awarded by the University Grants Commission, Government of India. Also, the authors are thankful to Mr. Amit Kumar Debnath, Research Scholar, Department of Mathematical Sciences, IIT (BHU) Varanasi, India, for his valuable suggestions on the present work.

				\section*{Funding}
				Not applicable

				\subsection*{Author Contributions}
				All authors contributed to the study conception and analysis. Material preparation and analysis was performed by Ram Surat Chauhan. The first draft of the manuscript was written by Ram Surat Chauhan and all authors commented on previous versions of the manuscript. All authors read and approved the final manuscript.
				
				\subsection*{Conflicts of interest/Competing interests} 
				The authors declare that they have no known competing financial interests or personal relationships that could have appeared to influence the work reported in this paper.

				\subsection*{Availability of data and material}
				Not applicable

				\subsection*{Code availability}
				Not applicable

			\end{document}